\documentclass[12pt]{amsart}
\usepackage{txfonts}      
\usepackage{amssymb}
\usepackage{eucal}
\usepackage{amsmath}
\usepackage{amscd}
\usepackage{xcolor}
\usepackage{multicol}
\usepackage[all]{xy}           
\usepackage{graphicx}
\usepackage{color}
\usepackage{colordvi}
\usepackage{xspace}
\usepackage{tikz}
\usepackage{makecell}
\usepackage{appendix}
\usepackage{enumerate,enumitem}
\usepackage{amsthm}
\usepackage[thicklines]{cancel}
\usepackage[compress]{cite}
\usepackage{ifpdf}
\ifpdf
\usepackage[colorlinks,final,backref=page,hyperindex]{hyperref}
\else
\usepackage[colorlinks,final,backref=page,hyperindex,hypertex]{hyperref}
\fi

\usepackage[active]{srcltx} 
\newtheorem{thm}{Theorem}[section]
\newtheorem{lem}[thm]{Lemma}
\newtheorem{cor}[thm]{Corollary}
\newtheorem{pro}[thm]{Proposition}
\theoremstyle{definition}
\newtheorem{defi}[thm]{Definition}
\newtheorem{ex}[thm]{Example}

\title[extending structures,pre-Poisson algebras] {Extending structures for pre-Poisson algebras and pre-Poisson bialgebras}

\author{Qianwen Zhu}
\address{Department of Mathematics, Zhejiang University of Science and Technology, Hangzhou, 310023}
\email{2630583032@qq.com}
\author{Guilai Liu}
\address{Chern Institute of Mathematics \& LPMC, Nankai University, Tianjin 300071, China}
\email{liugl@mail.nankai.edu.cn}
\author{Qinxiu Sun}
\address{Department of Mathematics, Zhejiang University of Science and Technology, Hangzhou, 310023} \email{qxsun@126.com}

\subjclass[2020]{17B62, 17B63, 17A30, 16T25}

\keywords{extending structures, pre-Poisson algebras}

\begin{document}
\begin{abstract}
	
In this paper, we explore the extending structures problem by the unified product for pre-Poisson algebras. 
In particular, the crossed product
and the factorization problem are investigated. Furthermore, a special case of extending structures is 
studied under the case of pre-Poisson algebras, which leads to the discussion of bicrossed products
and matched pairs of pre-Poisson algebras. We develop a bialgebra theory for pre-Poisson algebras and
establish the equivalence between matched pairs and pre-Poisson bialgebras.
We study coboundary
pre-Poisson bialgebras, which lead  to the introduction of the
pre-Poisson Yang-Baxter equation (PPYBE). A symmetric
solution of the PPYBE naturally gives a coboundary pre-Poisson
bialgebra.

\end{abstract}

\maketitle

\vspace{-1.2cm}

\tableofcontents

\vspace{-1.2cm}

\allowdisplaybreaks

\section{Introduction}
 The notion of Zinbiel algebras first appeared in Loday$'$s study of the algebraic structure behind the cup
product on the cohomology groups of a Leibniz algebra \cite{8}, and further was studied in \cite{9,10}. 
 Zinbiel algebras are Koszul dual to Leibniz algebras, and they are  special dendriform algebras.
 The anti-commutator of a Zinbiel algebra gives a commutative associative algebra, and the commutator of a pre-Lie algebra gives a Lie algebra.
Conversely, $\mathcal{O}$-operators on commutative associative algebras and Lie algebras give rise to 
Zinbiel algebras and pre-Lie algebras respectively. 
As we all know, a Poisson algebra
is both a Lie algebra and an associative algebra with the Leibniz rule.
Aguiar combined structures of Zinbiel algebra and  pre-Lie algebra on the same vector space, leading to the notion of a pre-Poisson algebra\cite{1}, which closely relates to the Poisson algebra.
On the one hand, a pre-Poisson algebra gives rise to a
Poisson algebra naturally through the sub-adjacent commutative associative algebra of the Zinbiel
algebra and the sub-adjacent Lie algebra of the pre-Lie algebra.
 On the other hand, from an $\mathcal{O}$-operator on a Poisson algebra, one can get a pre-Poisson algebra. 
  
The extending structures problem (ES-problem) arose in the study of group theory developed by 
Agore and Militaru \cite{2}, which unified the two well-known problems in group theory
–H$\rm{\ddot{o}}$lder's extension problem \cite{06} and the factorization problem of Ore \cite{32}. 
Since then, this theory was extended to various kinds of algebras, such as Lie algebras (bialgebras), 
Lie (associative) conformal algebras, left symmetric algebras (bialgebras), Leibniz algebras,
Zinbiel algebras, Poisson algebras (bialgebras), perm-algebras and Jordan algebras, 
see \cite{2,3,4,5,6,17,19,08,22,20,28,21,23,24,25,27,100,101} and references therein.
 But the extending structure theory for pre-Poisson algebras is still absent. This is the first motivation for writing this paper.
More precisely, we are devoted to investigating the ES-problem for pre-Poisson algebras:
 Assume that $A$ is a pre-Poisson algebra and $E$ is a vector space containing $A$ as a subspace. 
Characterize and classify all pre-Poisson algebraic structures on $E$ such 
that $A$ is a subalgebra of $E$ up to an isomorphism of pre-Poisson algebras whose restriction on $A$ is the identity map.

On the other hand, a bialgebraic structure on a given algebraic structure is obtained by a 
comultiplication together with compatibility conditions between the
multiplications and comultiplications. Lie bialgebras were introduced in the early 1980s by Drinfeld \cite{010}, which
have also been applied in the theory of classical integrable systems and are closely
related to the classical Yang-Baxter equation. In \cite{012,013}, V. Zhelyabin developed Drinfeld’s ideas to
introduce the notion of associative D-bialgebras, which are associative analogue of Lie bialgebras. Later, Aguiar developed
the notion of a balanced infinitesimal bialgebra \cite{011}, and an
antisymmetric infinitesimal bialgebra was considered in \cite{02}. Bai and Ni combined the ideas of 
Lie bialgebras and (commutative and cocommutative)
infinitesimal bialgebras in a uniform way, which leaded to the notion of a Poisson bialgebra \cite{04}. In \cite{09}, 
 a noncommutative Poisson bialgebra was developed. It is natural to combine Zinbiel bialgebras 
 and pre-Lie bialgebras \cite{01} in some certain way.
 This is another motivation for writing this paper. In detail,
 we develop a notion of a
pre-Poisson bialgebra, which  is both a Zinbiel bialgebra and a
pre-Lie bialgebra satisfying
some compatible conditions. Then we establish the equivalence between pre-Poisson bialgebras
and certain matched pairs of pre-Poisson algebras.
The study of coboundary pre-Poisson
bialgebras leads to the introduction of pre-Poisson Yang-Baxter
equation (PPYBE). A symmetric solution of the pre-Poisson Yang-Baxter equation naturally gives a 
(coboundary) pre-Poisson bialgebra.

 The paper is organized as follows. In Section 2, we study the ES-problem by the unified product for pre-Poisson algebras. 
   By means of unified product, we show that any pre-Poisson structure on $E$ 
such that $A$ is a pre-Poisson subalgebra is isomorphic to a unified product of $A$ by $V$, where $V$ is a complement of $A$ in $E$. With this tool, we construct an object $H^2_{p}(V, A)$ to give a theoretical answer for the ES-problem of pre-Poisson algebras.
 In Section 3, we discuss some special cases of unified products: crossed products, bicrossed products and flag extending structures. 
 In Section 4, we explore a bialgebra theory for pre-Poisson algebras and
establish the equivalence between matched pairs and pre-Poisson bialgebras. The study of coboundary
pre-Poisson bialgebras leads to the investigation of the pre-Poisson Yang-Baxter equation (PPYBE).
 A symmetric solution of the PPYBE naturally gives a coboundary pre-Poisson bialgebra. 

Throughout the paper, $k$ is a field.  All vector spaces and algebras are over $k$. 
 All algebras are finite-dimensional, although many results still hold in the infinite-dimensional case.
 
\section{Unified products for pre-Poisson algebras}

In this section, we study unified products for pre-Poisson algebras. We start with recalling 
some basic knowledge and extending structures for 
Zinbiel algebras and pre-Lie algebras. One can refer to \cite{01,02,4,5,6} for these facts.

A {\bf Zinbiel algebra} is a vector space $A$ together with a bilinear map $\ast:A\times A\longrightarrow A$ such that
 \begin{equation*}\label{za eq1.1}x\ast(y\ast z)=(y\ast x)\ast z+(x\ast y)\ast z,
 \end{equation*} 
 for all $x,y,z\in A$.

 A {\bf dendriform algebra } is a vector space $A$ together with two bilinear maps $\succ,\prec:A\times A\longrightarrow
 A$ such that
 \begin{equation*}(x\prec y)\prec z=x\prec (y\prec z)+x\prec (y\succ z)
,\end{equation*} 
\begin{equation*}x\succ( y\prec z)=(x\succ y)\prec z
,\end{equation*}
 \begin{equation*}x\succ( y\succ z)=(x\prec y)\succ z+(x\succ y)\succ z
,\end{equation*} 
for all $x,y,z\in A$.
Thus a Zinbiel algebra $(A,\ast)$ is indeed a dendriform algebra $(A,\succ,\prec)$ with $x\ast y=x\succ y=y\prec x$.

Let $(A,\ast)$ be a Zinbiel algebra. Define a binary operation
$\star:A\times A\longrightarrow
 A$ by   $$x\star y=x\ast
 y+y\ast
 x,$$ then $(A,\star)$ is a commutative associative algebra, which is called the {\bf associated commutative associative algebra} of $(A,\ast)$,
 denoted by $(\mathcal{C}(A),\star)$ or simply by
 $\mathcal{C}(A)$. We call
 $(A,\ast)$ a {\bf compatible Zinbiel
 algebra} structure on the commutative associative algebra 
$(\mathcal{C}(A),\star)$.  

A non-degenerate antisymmetric bilinear form $\omega$ on an associative algebra
$(A,\star)$ is called a {\bf Connes cocycle} if
 \begin{equation*}\omega(x\star y, z) + \omega(y \star z, x) + \omega(z\star x, y) = 0, ~~\forall~x, y, z\in A.\end{equation*} 
 
\begin{pro} Let $(A,\star)$ be a commutative associative algebra and $\omega$ be a Connes cocycle. 
Then there is a compatible Zinbiel algebra structure $\ast$ on $A$ defined by
\begin{equation*}\omega(x\ast y, z) = \omega(y,x \star z), ~~\forall~x, y, z\in A.\end{equation*} 
\end{pro}

\begin{defi} A {\bf representation} of a Zinbiel algebra $(A,\ast)$ is a triple $(V,l, r)$, where $V$ is a
vector space and $l,r : A \longrightarrow \hbox{End}
(V)$ are two linear maps satisfying the following relations, for all
$x,y\in A$,
\begin{equation*}\label{rza eq1.1}l(x)l(y)=l(x\ast y)+l(y\ast x)=l(y)l(x),\end{equation*}
\begin{equation*}\label{rza eq1.2}l(x)r(y)=r(x\ast y)=r(y)r(x)+r(y)l(x).\end{equation*}
\end{defi}

Let $A$ and $V$ be vector spaces. For
a linear map $f: A \longrightarrow \hbox{End} (V)$, define a linear
map $f^{*}: A \longrightarrow \hbox{End} (V^{*})$ by $\langle
f^{*}(x)u^{*},v\rangle=-\langle u^{*},f(x)v\rangle$ for all $x\in A,
u^{*}\in V^{*}, v\in V$, where $\langle \ , \ \rangle$ is the usual
pairing between $V$ and $V^{*}$.

\begin{pro} \label{zr} Let $(V,l,r)$ be a
representation of a Zinbiel algebra $(A,\ast)$. Then
\begin{enumerate}
	\item $(V^{*},-l^{*}-r^{*},r^{*})$ is also a
	representation of $(A,\ast)$. We call it the {\bf dual representation} of $(V,l,r)$.
	\item $(V,l+r)$ is a representation of the associated
	commutative associative algebra $(\mathcal{C}(A),\star)$.
	\item $(V,l)$ is a representation of the associated
	commutative associative algebra $(\mathcal{C}(A),\star)$.
\end{enumerate}
\end{pro}

\begin{ex} Let $(A,\ast)$ be a Zinbiel algebra,
	and
$L_{\ast}, R_{\ast} : A\longrightarrow \hbox{End} (A)$ be the
linear maps defined by $L_{\ast}(x)(y)=R_{\ast}(y)(x)=x\ast y$ for all $x,y\in A$. Then $(A,L_{\ast}, R_{\ast})$ is a representation of $(A,\ast)$,
which is called the {\bf regular representation} of $(A,\ast)$. 
Moreover, $(A^{*},-L_{\star}^{*}=-L^{*}_{\ast}-R^{*}_{\ast},
R_{\ast}^{*})$ is the dual representation of $(A,L_{\ast}, R_{\ast})$.
\end{ex}

A {\bf pre-Lie algebra} is a vector space $A$ together with a bilinear map $\circ:A \times A \longrightarrow A$ such that
	\begin{align*}
		x\circ(y\circ z)-(x\circ y)\circ z=y\circ(x\circ z)-(y\circ x)\circ z,
	\end{align*} 
for all $x,y,z\in A$. Define 
$[x, y]=x\circ y -y \circ x$ for all $x, y\in A$. Then $(A, [\ , \ ])$ is a Lie algebra, which is called the {\bf sub-adjacent Lie algebra} of $(A,\circ)$ and is denoted by $\frak g(A)$. Moreover, $(A,\circ)$ is called a {\bf compatible pre-Lie algebra} of $(A, [\ , \ ])$.

A Lie algebra $(A, [ \ , \ ])$ is called {\bf symplectic} if there exists a non-degenerate antisymmetric 2-cocycle $\omega$ on $A$, 
that is,
\begin{equation*}
	\omega([x, y], z) + \omega([y, z], x) + \omega([z, x], y) = 0, ~~\forall~x, y, z\in A. 
\end{equation*}
In this case, we call $\omega$ a {\bf symplectic} form on the Lie algebra $(A, [ \ , \ ])$.

\begin{pro}
Let $(A, [ \ , \ ],\omega)$ be a symplectic Lie algebra. Then there exists a
compatible pre-Lie algebra structure $\circ$ on $A$ given by
\begin{equation*}\omega(x \circ y, z)=-\omega(y, [x, z]), ~~\forall~x, y, z\in A. \end{equation*}
\end{pro}

\begin{defi}
Let $A$ be a pre-Lie algebra and $V$ a vector space. Assume that $\rho,\mu:A\longrightarrow End(V)$
are two linear maps. $(V,\rho,\mu)$ is called
a {\bf bimodule} of $A$ if
\begin{equation*}\rho(x)\rho(y)-\rho(x\circ y) = \rho(y)\rho(x) -\rho(y\circ x), \end{equation*}
\begin{equation*}\rho(x)\mu(y)-\mu(y)\rho(x) = \mu(x\circ y) - \mu(y)\mu(x), \end{equation*}
for all $x,y\in A$.
\end{defi}

 \begin{pro} \label{pl} Let $(V,\rho,\mu)$ be a
 	representation of a pre-Lie algebra $(A,\circ)$. Then
 	\begin{enumerate}
 		\item (i) $(V^{*}, l^{*}-r^{*},-r^{*})$ is also a
 		representation of $(A,\circ)$. We call it the {\bf dual representation} of $(V,\rho,\mu)$.
 		\item  $(V,l-r)$ is a representation of the sub-adjacent Lie algebra  $(A,[\ ,\ ])$.
 		\item  $(V,l)$ is a representation of the sub-adjacent Lie algebra  $(A,[\ ,\ ])$.
 	\end{enumerate}
 \end{pro}
 
 \begin{ex} Let $(A,\circ)$ be a pre-Lie algebra,
 	and
 	$L_{\circ}, R_{\circ} : A\longrightarrow \hbox{End} (A)$ be the
 	linear maps defined by $L_{\circ}(x)(y)=R_{\circ}(y)(x)=x\ast y$ for all $x,y\in A$. Then $(A,L_{\circ}, R_{\circ})$ is a representation of $(A,\circ)$,
 	which is called the {\bf regular representation} of $(A,\circ)$. 
 	Moreover, $(A^{*},\mathrm{ad}^{*}= L^{*}_{\circ}-R^{*}_{\circ},
 	-R_{\circ}^{*})$ is the dual representation of $(A,L_{\circ}, R_{\circ})$.
 \end{ex}

Now we come to study the representation theory of pre-Poisson algebras.

\begin{defi} \cite{1}
	A {\bf pre-Poisson algebra} is a triple $(A,
\ast,\circ )$, where $(A, \ast )$ is a Zinbiel algebra,
$(A,\circ )$ is a pre-Lie algebra and the following conditions hold:
\begin{equation}\label{ppa eq1.1}(x\circ y-y \circ x)\ast z=x\circ (y\ast z)-y\ast( x\circ
z),\end{equation}
\begin{equation}\label{ppa eq1.2}(x\ast y+y \ast x)\circ z=x\ast (y\circ z)+y\ast( x\circ
z),\end{equation} for all $x,y,z\in A$.\end{defi}

Let $(A, \ast,\circ )$ be a pre-Poisson algebra. Define
	$$x\star y=x\ast y+y\ast x,~[x,y]=x\circ y-y\circ x,~\forall~x,y\in A.$$ Then
	$(A, \star,[ \ , \ ] )$ is a Poisson algebra, that is, $(A,\star)$ is a commutative associative algebra, $(A,[\ .\ ])$ is  a Lie algebra and the following equation holds:
	\begin{eqnarray}
		[z,x\star y]=[z,x]\star y+x\star [z,y],\;\forall x,y,z\in A.
	\end{eqnarray}
	Moreover, $(A, \star,[ \ , \ ] )$ is called the {\bf sub-adjacent Poisson algebra} of $(A,\ast,\circ)$, and $(A,\ast,\circ)$ is called a {\bf compatible pre-Poisson algebra} of $(A, \star,[ \ , \ ] )$.
\begin{ex} Let $A$ be a vector space with a basis $\{e_1, e_2\}$. Define two bilinear
	maps $\ast,\circ:A \times A \longrightarrow A$ respectively by
	\begin{equation*}
		e_1\ast e_1=ae_2,~e_1\ast e_2=0,~e_2\ast e_1=0,~e_2\ast e_2=0,
	\end{equation*}
	\begin{equation*}
		e_1\circ e_1=be_1+ce_2,~e_1\circ e_2=be_2,~e_2\circ e_1=be_2,~e_2\circ e_2=0,
	\end{equation*}
	where $a,b,c\in k$. By a direct computation, $(A,\ast,\circ)$ is a pre-Poisson algebra.
\end{ex}
	
	\begin{defi} 
		\begin{enumerate}
			\item \cite{04}	Let $(A,\star,[\ ,\ ])$ be a Poisson algebra,
			$V$ be a vector space and $f,g: A
			\longrightarrow \hbox{End} (V)$ be linear maps. Then
			$(V,f,g)$ is called a
			{\bf representation} of $ (A,\star,[\ ,\ ])$ if $(V,f)$ is
			a representation of $(A,\star)$, $(V,g)$ is a representation of
			$(A,[\ ,\ ])$ and they satisfy the following compatible
			conditions:
			\begin{eqnarray}
				&&g(x\star y)=f(y)g(x)+f(x)g(y),\\
				&& f([x,y])=g(x)f(y)-f(y)g(x),\;\forall x,y\in A.
			\end{eqnarray}
		\item 	Let $(A,\ast,\circ)$ be a pre-Poisson algebra,
		$V$ be a vector space and $l,r,\rho,\mu: A
		\longrightarrow \hbox{End} (V)$ be four linear maps. Then
		$(V,l, r,\rho,\mu)$ is called a
		{\bf representation} of $(A,\ast,\circ)$ if $(V,l,r)$ is
		a representation of $(A,\ast)$, $(V,\rho,\mu)$ is a representation of
		$(A,\circ)$ and they satisfy the following compatible
		conditions:
		\begin{equation}\label{rppa 1.1}l(x\circ y-y\circ x)=\rho(x)l(y)-l(y)\rho(x)
			,\end{equation}
		\begin{equation}\label{rppa 1.2}\rho(x\ast y+y\ast x)=l(x)\rho(y)+l(y)\rho(x)
			,\end{equation}
		\begin{equation}\label{rppa 1.3}r(x)(\rho(y)-\mu(y))=l(y)\mu(x)-\mu(y\ast x),\end{equation}
		\begin{equation}\label{rppa 1.4}\mu(x)(l(y)+r(y))=r(y\circ x)+l(y)\mu(x),\end{equation}
		\begin{equation}\label{rppa 1.5}r(x)(\rho(y)-\mu(y))=\rho(y)r(x)-r(y\circ x),\end{equation}
		for all $x,y\in A$.
		\end{enumerate}
	 \end{defi}
By (\ref{rppa 1.1}) and (\ref{rppa 1.2}), we obtain
\begin{equation}\label{rppa 1.6}\rho(x\ast y+y\ast
x)=\rho(y)l(x)+\rho(x)l(y) .\end{equation}

By (\ref{rppa 1.3}) and (\ref{rppa 1.5}), we get
\begin{equation}\label{rppa 1.7}r(y\circ x)-\mu(y\ast
x)= \rho(y)r(x)-l(y)\mu(x).\end{equation}

\begin{pro} Let $(A,\ast,\circ)$ be a pre-Poisson algebra
 and $(V,l, r,\rho,\mu)$ be a representation of
$(A,\ast,\circ)$. Then
\begin{enumerate}
	\item $(V^{*},-l^{*}-r^{*},r^{*},\rho^{*}-\mu^{*},-\mu^{*})$
	is also a representation of $(A,\ast,\circ)$. We call it the {\bf dual
		representation} of $(V,l, r,\rho,\mu)$.
	\item $(V,l+r,\rho-\mu)$ 
	is a representation of $(A,\star,[\ ,\ ])$.
	\item $(V,l ,\rho )$ 
	is a representation of $(A,\star,[\ ,\ ])$.
\end{enumerate}
\end{pro}

\begin{proof} In the light of Propositions
\ref{zr} and \ref{pl}, $(V^{*},-l^{*}-r^{*},r^{*})$ is a
representation of $(A,\ast)$ and
$(V^{*},\rho^{*}-\mu^{*},-\mu^{*})$ is  a
representation of $(A,\circ)$. We only need to check that (\ref{rppa 1.1})-(\ref{rppa 1.5}) hold for
$(V^{*},-l^{*}-r^{*},r^{*},\rho^{*}-\mu^{*},-\mu^{*})$.
As an example, we give an explicit proof of (\ref{rppa 1.2}). Thanks
to Eqs. (\ref{rppa 1.2}), (\ref{rppa 1.4}), (\ref{rppa 1.6}) and (\ref{rppa 1.7}),
we have for any $x,y\in A,u\in V,v^{*}\in V^{*} $,
\begin{eqnarray*}&&\langle-(l^{*}(x)+r^{*}(x))(\rho^{*}(y)-\mu^{*}(y))v^{*}
-(l^{*}(y)+r^{*}(y))(\rho^{*}(x)-\mu^{*}(x))v^{*}
\\&&-(\rho^{*}-\mu^{*})(x\ast y+y\ast x)v^{*},u\rangle\\&=&
\langle
-l^{*}(x)\rho^{*}(y)v^{*}-r^{*}(x)\rho^{*}(y)v^{*}+l^{*}(x)\mu^{*}(y)v^{*}
+r^{*}(x)\mu^{*}(y)v^{*} -
l^{*}(y)\rho^{*}(x)v^{*}\\&&-r^{*}(y)\rho^{*}(x)v^{*}+l^{*}(y)\mu^{*}(x)v^{*}
+r^{*}(y)\mu^{*}(x)v^{*}
-(\rho^{*}-\mu^{*})(x\ast y+y\ast x)v^{*},u\rangle
\\&=&\langle
v^{*},-\rho(y)(l(x)+r(x))u+\mu(y)(l(x)+r(x))u
-\rho(x)(l(y)+r(y))u\\&&+
\mu(x)(l(y)+r(y))u+\rho(x\ast y+y\ast
x)u-\mu(x\ast y+y\ast x)u\rangle\\& =&0,\end{eqnarray*} which
indicates that (\ref{rppa 1.2}) holds for substituting $(l,\rho)$ by 
$(-l^{*}-r^{*},\rho^{*}-\mu^{*})$.
\end{proof}

\begin{ex} Let $(A,\ast,\circ)$ be a pre-Poisson algebra.
Then $(A,L_{\ast}, R_{\ast},L_{\circ}, R_{\circ})$ is a
representation of $(A,\ast,\circ)$, which is called the {\bf regular
representation} of $(A,\ast,\circ)$. Moreover, 
$(A^{*},-L_{\star}^{*},
R_{\ast}^{*},\mathrm{ad}^{*}$,
$-R_{\circ}^{*})$ is the  dual
representation of $(A,L_{\ast}, R_{\ast},L_{\circ}, R_{\circ})$.
Furthermore, $(A,  {L} _{\ast},  {L} _{\circ})$  and 
 $(A^{*},- {L}^{*}_{\ast},  {L}^{*}_{\circ})$ are representations of $(A,\star,[\ ,\ ])$.
\end{ex}

\begin{defi} \cite{5}
		Let $(A,\ast_{1})$ be a Zinbiel algebra and $V$ be a vector space.
 A  {\bf Zinbiel extending datum of $(A,\ast_{1})$ through $V$} 
 is a system $\Omega(A, V) =(l_{1},r_{1},l_{2},r_{2},f,\ast_{2})$ consisting of four linear maps 
 $ l_{1},r_{1}: A \longrightarrow End(V), ~~ l_{2},r_{2} : V \longrightarrow End(A)$ and two bilinear maps
	$f: V\times V \longrightarrow A,~~\ast_{2}: V \times V \longrightarrow V.$
	Denote by $A\natural V $ the vector space $ A \oplus V$ with the bilinear map $\ast:(A\oplus V)\times(A \oplus V)\longrightarrow A \oplus V$ 
defined by
	\begin{align}
		(a,x)\ast(b,y)=(a\ast_{1}b+l_{2}(x)b+r_{2}(y)a+f(x,y),x\ast_{2}y+l_{1}(a)y+r_{1}(b)x),\label{4}
	\end{align}
	for all $a,b\in A$ and $x,y\in V$. The object $A\natural V $ is
 called the \textit{unified product} of $(A,\ast_{1})$ and $V$ if $(A \oplus V,\ast)$ is a Zinbiel algebra.
In this case, the Zinbiel extending datum $\Omega(A, V)$
     is called a Zinbiel extending structure of $(A,\ast_{1})$ through $V$.
	\end{defi}

	\begin{thm}\cite{5}\label{22}
	Let $(A,\ast_{1})$ be a Zinbiel algebra, $V$ be a vector space and $\Omega(A, V) =(l_{1},r_{1},l_{2},r_{2},f,\ast_{2})$ be a 
\textit{Zinbiel extending datum of $(A,\ast_{1})$ through $V$}. 
Then $A\natural V$ is a \textit{unified product} of $(A,\ast_{1})$ and $V$
 if and only if the following conditions hold for all $a,b\in A$ and $x,y,z\in V$,
		\begin{equation}\label{za1}
			(V,l_{1},r_{1})~\hbox{ is a representation of a Zinbiel algebra}~ (A,\ast_{1}),	\end{equation}
			\begin{equation}\label{za2}l_{1}(a)(x\ast_{2}y)=(l_{1}(a)x)\ast_{2}y+(r_{1}(a)x)\ast_{2}y+l_{1}(l_{2}(x)a)y+l_{1}(r_{2}(x)a)y,\end{equation}			\begin{equation}\label{za3}x\ast_{2}(l_{1}(a)y)+r_{1}(r_{2}(y)a)x=(l_{1}(a)x)\ast_{2}y+(r_{1}(a)x)\ast_{2}y+l_{1}(l_{2}(x)a)y+l_{1}(r_{2}(x)a)y,\end{equation}
			\begin{equation}\label{za4}x\ast_{2}(r_{1}(a)y)+r_{1}(r_{2}(y)a)x=r_{1}(a)(y\ast_{2}x)+r_{1}(a)(x\ast_{2}y),\end{equation}
			\begin{equation}\label{za5}a\ast_{1}(r_{2}(x)b)+r_{2}(l_{1}(b)x)a=r_{2}(x)(b\ast_{1}a+a\ast_{1}b),\end{equation}
\begin{equation}\label{za6}a\ast_{1}(l_{2}(x)b)+r_{2}(r_{1}(b)x)a=(l_{2}(x)a)\ast_{1}b+(r_{2}(x)a)\ast_{1}b+l_{2}(l_{1}(a)x)b+l_{2}(r_{1}(a)x)b,\end{equation}
			\begin{equation}\label{za7}l_{2}(x)(a\ast_{1}b)=(l_{2}(x)a)\ast_{1}b+(r_{2}(x)a)\ast_{1}b+l_{2}(l_{1}(a)x)b+l_{2}(r_{1}(a)x)b,\end{equation}
			\begin{equation}\label{za8}a\ast_{1}f(x,y)+r_{2}(x\ast_{2}y)a=r_{2}(y)(l_{2}(x)a)+r_{2}(y)(r_{2}(x)a)+f(l_{1}(a)x,y)+f(r_{1}(a)x,y),\end{equation}
			\begin{equation}\label{za9}l_{2}(x)(r_{2}(y)a)+f(x,l_{1}(a)y)=r_{2}(y)(r_{2}(x)a)+r_{2}(y)(l_{2}(x)a)+f(l_{1}(a)x,y)+f(r_{1}(a)x,y),\end{equation}
			\begin{equation}\label{za10}l_{2}(x)(l_{2}(y)a)+f(x,r_{1}(a)y)=l_{2}(x\ast_{2}y)a+l_{2}(y\ast_{2}x)a+f(x,y)\ast_{1}a+f(y,x)\ast_{1}a,\end{equation}
			\begin{equation}\label{za11}l_{2}(x)f(y,z)+f(x,y\ast_{2}z)=r_{2}(z)f(x,y)+r_{2}(z)f(y,x)+f(x\ast_{2}y,z)+f(y\ast_{2}x,z),\end{equation}
			\begin{equation}\label{za12}x\ast_{2}(y\ast_{2}z)+r_{1}(f(y,z))x=(x\ast_{2}y)\ast_{2}z+(y\ast_{2}x)\ast_{2}z+l_{1}(f(x,y))z+l_{1}(f(y,x))z.
		\end{equation}
	\end{thm}

Assume that $\Omega(A,V)=(l_{1},r_{1},l_{2},r_{2},f,\ast_{2})$ 
    is a Zinbiel extending datum of $A$ through $V$. If $f$ is trivial, in view of Theorem \ref{22}, we get that $\Omega(A,V)$
 is a Zinbiel extending structure of $A$ through $V$ if and only if $(V, \ast_2)$ is a Zinbiel algebra,
$(A,l_2,r_2) $ is a bimodule of $V$, $(V,l_{1},r_{1})$ is a bimodule of $A$
and they satisfy Eqs. (\ref{za2})-(\ref{za7}). In this case, the associated
unified product
$A\natural V$ is denoted by $A \bowtie V$ , which is called the bicrossed product of $A$ by $V$ and $(A,V,l_{1},r_{1},l_{2},r_{2})$ is
called the matched pair of Zinbiel algebras.

In the sequel, we recall the extending structures for pre-Lie algebras.
	
\begin{defi} \cite{6}
	Let $(A,\circ_{1})$ be a pre-Lie algebra and $V$ be a vector space. A \textbf{pre-Lie extending datum of $(A,\circ_{1})$ through $V$}  is a system $\Omega(A, V) =(\rho_1,\mu_1,\rho_2,\mu_2,g,\circ_{2})$ consisting of four linear maps 
\begin{align*}
	\rho_1,\mu_1: A \longrightarrow End(V), ~~~~\rho_2,\mu_2: V  \longrightarrow End(A) \end{align*}
and two bilinear maps
	\begin{align*}g: V\times V \longrightarrow A, ~\circ_{2}: V \times V \longrightarrow V.
    \end{align*}
     Denote by $A\natural V $ 
    the vector space $ A\oplus V$ equipped with the bilinear map $\circ:(A\oplus V)\times(A \oplus V) \longrightarrow A \oplus V$ 
    given by 
    \begin{align}
    	(a,x)\circ(b,y)=(a\circ_{1}b+\rho_2(x)b+\mu_2(y)a+g(x,y),x\circ_{2}y+\rho_1(a)y+\mu_1(b)x),\label{3}
    \end{align}
    for all $a,b\in A$ and $x,y\in V$. The object $A\natural V $ is called the \textit{unified product} of $(A,\circ_{1})$ and $V$ 
    if $(A\oplus V,\circ)$ is a pre-Lie algebra. In this case, the pre-Lie extending datum $\Omega(A, V)$
     is called a pre-Lie extending structure of $(A,\circ_{1})$ through $V$.
	\end{defi}

	\begin{thm} \cite{6}\label{21}
		Let $(A,\circ_{1})$ be a pre-Lie algebra, $V$ be a vector space and $\Omega(A, V) =(\rho_1,\mu_1,\rho_2,\mu_2,g,\circ_{2})$ be a \textit{pre-Lie extending datum of $(A,\circ_{1})$ through $V$}. Then $A\natural V$ is a \textit{unified product} of $(A,\circ_{1})$ and $V$ if and only if the following conditions hold for all $a,b\in A$ and $x,y,z\in V$,
		\begin{equation}\label{pra1}
		(V,\rho_{1},\mu_{1}) ~\hbox{is a bimodule of the pre-Lie algebra} ~(A,\circ_{1}),	\end{equation}
\begin{equation}\label{pra2} \rho_{1}(a)(x\circ_{2} y) =(\rho_{1}(a)x-\mu_1(a)x)\circ_{2} y+\rho_{1}(\mu_2(x)a-\rho_2(x)a)y+\mu_{1}(\mu_2(y)a)x+x\circ_{2}(\rho_1(a)y),\end{equation}
	\begin{equation}\label{pra3}\mu_{1}(a)(x\circ_{2}y-y\circ_{2}x)=\mu_{1}(\rho_{2}(y)a)x-\mu_{1}(\rho_{2}(x)a)y+x\circ_{2}(\mu_{1}(a)y)-y\circ_{2}(\mu_{1}(a)x),\end{equation}
		\begin{equation}\label{pra4} \rho_{2}(x)(a\circ_{1}b)=(\rho_2(x)a-\mu_2(x)a)\circ_{1}b+\rho_{2}(\mu_1(a)x-\rho_1(a)x)b+a\circ_{1}(\rho_2(x)b)+\mu_{2}(\mu_{1}(b)x)a,\end{equation}
		\begin{equation}\label{pra5}\mu_{2}(x)(a\circ_{1}b-b\circ_{1}a)=a\circ_{1}(\mu_2(x)b)+\mu_{2}(\rho_{1}(b)x)a-b\circ_{1}(\mu_2(x)a)-\mu_{2}(\rho_1(a)x)b,\end{equation}
		\begin{align}\label{pra6}\rho_{2}(x\circ_{2}y-y\circ_{2}x)a +(g(x, y)-g(y,x))\circ_{1}a
=&\rho_{2}(x)(\rho_2(y)a)-\rho_{2}(y)(\rho_2(x)a)\\
			&+g(x,\mu_1(a)y)-g(y,\mu_1(a)x),\nonumber\end{align}
		\begin{align}\label{pra7} 
			\mu_{2}(x\circ_{2} y)a+a\circ_{1} g(x, y)
			=&\rho_{2}(x)(\mu_{2}(y)a) +\mu_{2}(y)(\mu_2(x)a-\rho_2(x)a) \\
			&+g(\rho_{1}(a)x-\mu_1(a)x, y)+g(x, \rho_1(a)y),\nonumber\end{align}
		\begin{align}\label{pra8}g(x\circ_{2} y, z)-g(x, y\circ_{2} z)+\mu_{2}(z)g(x, y)-\rho_{2}(x)g(y, z)=&g(y\circ_{2}x, z)-g(y, x\circ_{2}z) \\
&+\mu_{2}(z)g(y, x)-\rho_{2}(y)g(x,z),\nonumber\end{align}
		\begin{align}\label{pra9}(x\circ_{2} y)\circ_{2} z-x\circ_{2}(y\circ_{2} z)+\rho_{1}(g(x, y))z-\mu_{1}(g(y, z))x
=&(y\circ_{2}x)\circ_{2} z-y\circ_{2}(x\circ_{2} z)\\
		&+\rho_{1}(g(y, x))z-\mu_{1}(g(x, z))y.\nonumber
		\end{align}
	\end{thm}

Let $\Omega(A,V)=(\rho_1,\mu_1,\rho_2,\mu_2,g,\circ_{2})$ 
    be a pre-Lie extending datum of $A$ through $V$. If $g$ is trivial, by Theorem \ref{21}, 
    we obtain that $\Omega(A,V)$
 is a pre-Lie extending structure of $A$ through $V$ if and only if $(V, \circ_2)$ is a pre-Lie algebra,
$(A,\rho_2,\mu_2) $ is a bimodule of $V$, $(V,\rho_{1},\mu_{1})$ is a bimodule of $A$
and they satisfy Eqs. (\ref{pra2})-(\ref{pra5}). In this case, the associated
unified product
$A\natural V$ is denoted by $A \bowtie V$, which is called the bicrossed product of $A$ by $V$ 
and $(A,V,\rho_{1},\mu_{1},\rho_{2},\mu_{2})$ is
called the matched pair of pre-Lie algebras. 

We are ready to study extending structures for pre-Poisson algebras.

\begin{defi}\label{6}
Let $(A,\ast_{1},\circ_{1})$ be a pre-Poisson algebra, $E$ a vector space such that $A$ is 
 a subspace of $E$ and $V$ is a complement of $A$ in $E$.
 For a linear map $\varphi:E\longrightarrow E$ we consider the following diagram:
 \begin{equation}\label{E} \xymatrix{
A\ar@{=}[d] \ar[r]^{i} & E\ar[d]_{\varphi} \ar[r]^{\pi} & V \ar@{=}[d] \\
A \ar[r]^{i} & E \ar[r]^{\pi} & V ,}\end{equation}
 where $\pi:E \longrightarrow V$ is the canonical projection of $E =A\oplus V$
  on $V$ and $i:A\longrightarrow E$ is
the inclusion map. We say that $\varphi:E\longrightarrow E$ stabilizes $A$ (resp. co-stabilizes $V$) if the
left square (resp. the right square) of the diagram (\ref{E}) is commutative.
  Let $(E,\ast,\circ)$ and $(E,\ast',\circ')$ be two pre-Poisson algebraic structures on 
  $E$ containing $(A,\ast_{1},\circ_{1})$ as a pre-Poisson subalgebra.
\begin{enumerate}
			\item $(E,\ast, \circ)$ and $(E,\ast', \circ')$ are called {\bf equivalent}, denoted 
  by $(E, \circ,\ast) \equiv (E, \ast',\circ')$, if there exists a pre-Poisson
   algebra isomorphism $\phi: (E, \ast,\circ)\longrightarrow (E,\ast', \circ')$ which stabilizes $A$, i.e.
   whose restriction on $A$ is the identity map.
 
\item $(E,\ast, \circ)$ and $(E, \ast',\circ')$ are called {\bf cohomologous},
 denoted by $(E,\ast, \circ) \approx (E,\ast', \circ')$, 
 if there exists a pre-Poisson algebra isomorphism $\phi: (E,\ast, \circ) \longrightarrow (E,\ast', \circ')$ 
    which stabilizes $A$ and
co-stabilizes $V$, namely, whose restriction on $A$ and $V$ are the identity maps.
		\end{enumerate}
	\end{defi}

It is easy to see that $\equiv$ (resp. $\approx$) is an equivalent relation on the set of 
all pre-Poisson structures on $E$ containing $(A,\ast,\circ)$ as a subalgebra. Denote
the set of all equivalent classes via $\equiv$ (resp. $\approx$) by $Extd(E, A)$ (resp. $Extd'(E,A)$). 
Thus, we only need to characterize $Extd(E, A)$~(resp. $Extd'(E,A))$ for investigating the
ES-problem.

\begin{defi}
  		Let $(A,\ast_{1},\circ_{1})$ be a pre-Poisson algebra and $V$ a vector space. 
  A \textbf{pre-Poisson extending datum of $(A,\ast_{1},\circ_{1})$ through $V$}
    is a system $\Omega(A,V)=(l_{1},r_{1},l_{2},r_{2},f,\ast_{2},\rho_1,\mu_1,\rho_2,\mu_2$,
    $g,\circ_{2})$ 
    consisting of eight linear maps 
  		\begin{align*}
  			l_{1},r_{1},\rho_{1},\mu_{1}: A \longrightarrow End(V), ~~ l_{2},r_{2},\rho_{2},\mu_{2} : V \longrightarrow End(A)
  			\end{align*}
  and four bilinear maps 
  \begin{align*} f,g: V\times V \longrightarrow A,  ~~\ast_{2},\circ_{2}: V \times V \longrightarrow V.
  		\end{align*}
  		Denote by $A\natural V $ the vector space $ A\oplus V$ 
  with the bilinear maps $\ast$ and $\circ$ : $(A\oplus V)\times(A \oplus V) \longrightarrow A 
  		\oplus V$ given by Eqs.~(\ref{4}) and (\ref{3}) respectively.
 The object $A\natural V $ is called the \textit{unified product} of $(A,\ast_{1},\circ_{1})$ and $V$ if 
  $(A\oplus V,\ast,\circ)$ is a pre-Poisson algebra. 
  In this case, the extending datum $\Omega(A, V) $ is called a pre-Poisson extending structure of $(A,\ast_{1},\circ_{1})$ through $V$.
  Denote by $\mathcal{A}(A, V)$ the set of all pre-Poisson extending structures of $(A,\ast_{1},\circ_{1})$ through $V$.
  	\end{defi}
  
  \begin{pro}\label{5} Let $(A,\ast_{1},\circ_{1})$ be a pre-Poisson algebra, $V$ be a vector space and $\Omega(A,V)=(l_{1},r_{1},l_{2},r_{2},f,\ast_{2},\rho_1,\mu_1,\rho_2,\mu_2,g,\circ_{2})$ be a \textit{pre-Poisson extending datum of $(A,\ast_{1},\circ_{1})$ through $V$}. Then $A\natural V$ is a unified product of $(A,\ast_{1},\circ_{1})$ and $V$ if and only if the following conditions hold for all $a,b\in A$ and $x,y,z\in V$,
  \begin{align} 
  &\text{$(l_{1},r_{1},l_{2},r_{2},f,\ast_{2})$ is a  \textit{Zinbiel extending structure of $(A,\ast_{1})$ through $V$},}\label{a40}\\
  &\text{$(\rho_1,\mu_1,\rho_2,\mu_2,g,\circ_{2})$ is a  \textit{pre-Lie extending structure of $(A,\circ_{1})$ through $V$},}\\   
  &l_{1}(a\circ_{1}b-b\circ_{1}a)x=\rho_{1}(a)(l_{1}(b)x)-l_{1}(b)(\rho_1(a)x),\label{a42}\\
  &r_{1}(b)(\rho_1(a)x-\mu_1(a)x)=\rho_{1}(a)(r_{1}(b)x)-r_{1}(a\circ_{1}b)x,\label{a43}\\ 
  &r_{1}(b)(\mu_1(a)x-\rho_1(a)x)=\mu_{1}(a\ast_{1}b)x-l_{1}(a)(\mu_1(b)x),\label{a44}\\
  &(\rho_1(a)x-\mu_1(a)x)\ast_{2}y+l_{1}(\mu_{2}(x)a-\rho_2(x)a)y=\rho_{1}(a)(x\ast_{2}y)-x\ast_{2}(\rho_1(a)y)\label{a45}\\
  &-r_{1}(\mu_{2}(y)a)x,\nonumber\\
  &(\mu_1(a)x-\rho_1(a)x)\ast_{2}y+l_{1}(\rho_2(x)a-\mu_{2}(x)a)y=x\circ_{2}(l_{1}(a)y)+\mu_{1}((r_{2}(y)a))x\label{a46}\\
  &-l_{1}(a)(x\circ_{2}y),\nonumber\\
  &r_{1}(a)(x\circ_{2}y-y\circ_{2}x)=x\circ_{2}(r_{1}(a)y)+\mu_{1}(l_{2}(y)a)x-y\ast_{2}(\mu_1(a)x)-r_{1}(\rho_2(x)a)y,\label{a47}\\
  &r_{2}(x)(a\circ_{1}b-b\circ_{1}a)=a\circ_{1}(r_{2}(x)b)+\mu_{2}(l_{1}(b)x)a-b\ast_{1}(\mu_{2}(x)a)-r_{2}(\rho_1(a)x)b,\label{a48}\\
  &(\mu_{2}(x)a-\rho_2(x)a)\ast_{1}b+l_{2}(\rho_1(a)x-\mu_1(a)x)b=a\circ_{1}(l_{2}(x)b)+\mu_{2}(r_{1}(b)x)a\label{a49}\\
  &-l_{2}(x)(a\circ_{1}b),\nonumber\\
  &(\rho_2(x)a-\mu_{2}(x)a)\ast_{1}b+l_{2}(\mu_1(a)x-\rho_1(a)x)b=\rho_{2}(x)(a\ast_{1}b)-a\ast_{1}(\rho_2(x)b)\label{a50}\\
  &-r_{2}(\mu_1(b)x)a,\nonumber\\
  &r_{2}(y)(\mu_{2}(x)a-\rho_2(x)a)+f(\rho_1(a)x-\mu_1(a)x,y)=a\circ_{1}f(x,y)+\mu_{2}(x\ast_{2}y)a\label{a51}\\
  \notag&-l_{2}(x)(\mu_{2}(y)a)-f(x,\rho_1(a)y),\\
  &r_{2}(y)(\rho_2(x)a-\mu_{2}(x)a)+f(\mu_1(a)x-\rho_1(a)x,y)=\rho_{2}(x)(r_{2}(y)a)+g(x,l_{1}(a)y)\label{a52}\\
  \notag&-a\ast_{1}(g(x,y))-r_{2}(x\circ_{2}y)a,\\
  &(g(x,y)-g(y,x))\ast_{1}a+l_{2}(x\circ_{2}y-y\circ_{2}x)b=\rho_{2}(x)(l_{2}(y)a)+g(x,r_{1}(a)y)\label{a53}\\
  \notag&-l_{2}(y)(\rho_2(x)a)-f(y,\mu_1(a)x),\\
  &r_{2}(z)(g(x,y)-g(y,x))+f(x\circ_{2}y-y\circ_{2}x,z)=\rho_{2}(x)f(x,y)+g(x,y\ast_{2}z)\label{a54}\\
  \notag&-l_{2}(y)g(x,z)-f(y,x\circ_{2}z),\\
  &(x\circ_{2}y-y\circ_{2}x)\ast_{2}z+l_{1}(g(x,y)-g(y,x))z=x\circ_{2}(y\ast_{2}z)-y\ast_{2}(x\circ_{2}z)\label{a55}\\
  \notag&+\mu_{1}(f(x,y))x-r_{1}(g(x,z))y,\\
  &\rho_{1}(a\ast_{1}b+b\ast_{1}a)x=l_{1}(a)(\rho_1(b)x)+l_{1}(b)(\rho_1(a)x),\label{a56}\\
  &\mu_{1}(b)(l_{1}(a)x+r_{1}(a)x)=l_{1}(a)(\mu_{1}(b)x)+r_{1}(a\circ_{1} b)x,\label{a57}\\
  &(l_{1}(a)x+r_{1}(a)x)\circ_{2}y+\rho_{1}(r_{2}(x)a+l_{2}(x)a)y=l_{1}(a)(x\circ_{2}y)+x\ast_{2}(\rho_1(a)y)+r_{1}(\mu_{2}(y)a)x,\label{a58}\\
  &\mu_{1}(a)(x\ast_{2}y+y\ast_{2}x)=x\ast_{2}(\mu_1(a)y)+y\ast_{2}(\mu_1(a)x)+r_{1}(\rho_2(y)a)x+r_{1}(\rho_2(x)a)y,\label{a59}\\
  &\mu_{2}(x)(a\ast_{1}b+b\ast_{1}a)=a\ast_{1}(\mu_2(x)b)+b\ast_{1}(\mu_2(x)a)+r_{2}(\rho_1(b)x)a+r_{2}(\rho_1(a)x)b,\label{a60}\\
  &(r_{2}(x)a+l_{2}(x)a)\circ_{1} b+\rho_{2}(l_{1}(a)x+r_{1}(a)x)b=a\ast_{1}(\rho_2(x)b)+r_{2}(\mu_1(b)x)a+l_{2}(x)(a\circ_{1}b),\label{a61}\\
  &\mu_{2}(y)(r_{2}(x)a+l_{2}(x)a)+g(l_{1}(a)x+r_{1}(a)x,y)=a\ast_{1}g(x,y)+r_{2}(x\circ_{2}y)a+l_{2}(x)(\mu_{2}(y)a)\label{a62}\\
  \notag&+f(x,\rho_1(a)y),\\
  &(f(x,y)+f(y,x))\circ_{1}a+\rho_{2}(x\ast_{2}y+y\ast_{2}x)a=l_{2}(x)(\rho_2(y)a)+l_{2}(y)(\rho_2(x)a)\label{a63}\\
  \notag&+f(x,\mu_1(a)y)+f(y,\mu_1(a)x),\\
  &\mu_{2}(z)(f(x,y)+f(y,z))+g(x\ast_{2}y+y\ast_{2}x,z)=l_{2}(x)g(y,z)+l_{2}(y)g(x,z)+f(x,y\circ_{2}z)\label{a64}\\
  \notag&+f(y,x\circ_{2} z),\\
  &(x\ast_{2} y+y\ast_{2}x)\circ_{2}z+\rho_{1}(f(x,y)+f(y,x)) z=x\ast_{2}(y\circ_{2}z)+y\ast_{2}(x\circ_{2}z)+r_{1}(g(y,z))x\label{a65}\\
  \notag&+r_{1}(g(x,z))y.
  \end{align}

  \end{pro}
  \begin{proof}
  	By Theorem \ref{22} and Theorem \ref{21}, $(A\oplus V,\ast)$ is a Zinbiel algebra 
  if and only if $(l_{1},r_{1},l_{2},r_{2},f,\ast_{2})$ is a 
  Zinbiel extending structure of $(A,\ast_{1})$ through $V$,
   and $(A\oplus V,\circ)$ is a pre-Lie algebra if 
   and only if $(\rho_1,\mu_1,\rho_2,\mu_2,g,\circ_{2})$ is
    a pre-Lie extending structure of $(A,\circ_{1})$ through $V$. So we only 
  need to check that
  	\begin{align}
  		&((a,x)\circ (b,y)-(b,y)\circ (a,x))\ast (c,z)=(a,x)\circ((b,y)\ast (c,z))-(b,y)\ast((a,x)\circ (c,z)),\label{60}\\
  		&((a,x)\ast(b,y)+(b,y)\ast(a,x))\circ(c,z)=(a,x)\ast((b,y)\circ(c,z))+(b,y)\ast((a,x)\circ(c,z)),\label{61}
  \end{align}
  	for all $a,b,c\in A,~x,y,z\in V$ if and only if \eqref{a42}-\eqref{a65} hold.
In fact, for the Eq. \eqref{60}, we consider it by the following cases:

(i) \eqref{60} holds for the triple $((a,0),(b,0),(0,x))$ if and only if \eqref{a42} and \eqref{a48} hold.

(ii) \eqref{60} holds for the triple $((a,0),(0,x),(b,0))$ if and only if \eqref{a43} and \eqref{a49} hold.

(iii) \eqref{60} holds for the triple $((0,x),(a,0),(b,0))$ if and only if \eqref{a44} and \eqref{a50} hold.

(iv) \eqref{60} holds for the triple $((a,0),(0,x),(0,y))$ if and only if \eqref{a45} and \eqref{a51} hold.

(v) \eqref{60} holds for the triple $((0,x),(a,0),(0,y))$ if and only if \eqref{a46} and \eqref{a52} hold.

(vi) \eqref{60} holds for the triple $((0,x),(0,y),(a,0))$ if and only if \eqref{a47} and \eqref{a53} hold.

(vii)\eqref{60} holds for the triple $((0,x),(0,y),(0,z))$ if and only if \eqref{a54} and \eqref{a55} hold.
  
For the Eq.\eqref{61}, we consider it by the following cases: 

(i) \eqref{61} holds for the triple $((a,0),(b,0),(0,x))$ if and only if \eqref{a56} and \eqref{a60} hold.

(ii) \eqref{61} holds for the triple $((a,0),(0,x),(b,0))$ if and only if \eqref{a57} and \eqref{a61} hold.

(iii) \eqref{61} holds for the triple $((a,0),(0,x),(0,y))$ if and only if \eqref{a58} and \eqref{a62} hold.

(iv) \eqref{61} holds for the triple $((0,x),(0,y),(a,0))$ if and only if \eqref{a59} and \eqref{a63} hold.

(v) \eqref{61} holds for the triple $((0,x),(0,y),(0,z))$ if and only if \eqref{a64} and \eqref{a65} hold. 
 
 By direct computations, we can prove that the above statements hold.
  \end{proof}
  \begin{thm}\label{11}
  	Let $(A,\ast_{1},\circ_{1})$ be a pre-Poisson algebra, $E$ a vector space containing $A$ 
  as a subspace and $(E,\ast,\circ)$ a pre-Poisson algebra such that $A$ 
  is a pre-Poisson subalgebra of $E$. Then there is a pre-Poisson \textit{extending structure} $\Omega(A,V)=(l_{1},r_{1},l_{2},r_{2},f,\ast_{2},\rho_1,\mu_1,\rho_2,\mu_2,g,\circ_{2})$ 
  of $A$ through a subspace $V$ of $E$ and an isomorphism of pre-Poisson algebras 
  $(E,\ast,\circ) \cong A\natural V$ 
  that
stabilizes $A$ and co-stabilizes $V$.
  \end{thm}
  \begin{proof}Clearly,
  there is a linear map $p : E \longrightarrow A$ such that $p(a)=a$, for all $a \in A$. 
  Then $V=\mathrm{ker}(p)$ is a subspace of $E$, which is a complement of $A$ in $E$. Define the linear maps 
  \begin{align*}
 	l_{1},r_{1},\rho_{1},\mu_{1}: A \longrightarrow End(V),~~l_{2},r_{2},\rho_{2},\mu_{2}: V  \longrightarrow End(A)
 \end{align*}
 and the bilinear maps
 \begin{align*}f,g: V\times V \longrightarrow A,~\ast_{2},\circ_{2}: V\times V \longrightarrow V \end{align*}
 respectively by
  \begin{align*}
 	&l_{1}(a)x=a\ast x-p(a\ast x),~r_{1}(a)x=x\ast a-p(x\ast a),~\rho_{1}(a)x=a\circ x-p(a\circ x),\\
 	&\mu_{1}(a)x=x\circ a-p(x\circ a),~l_{2}(x)a=p(x\ast a),~r_{2}(a)x=p(a\ast x),~ \rho_{2}(x)b=p(x\circ a),\\
 &\mu_{2}(x)a=p(a\circ x),~ f(x,y)=p(x\ast y),~g(x,y)=p(x\circ y),\\
 &x\ast_{2} y=x\ast y-p(x\ast y),~x\circ_{2} y=x\circ y-p(x\circ y),
 \end{align*}
 for all $a\in A$,~$x,y\in V$. Based on Theorem 3.10 \cite{6} and Theorem 1.8 \cite{5}, 
 it is easy to prove that
  $\Omega(A,V)=(l_{1},r_{1},l_{2},r_{2},f,\ast_{2},\rho_1,\mu_1,\rho_2,\mu_2,g,\circ_{2})$ 
 is a pre-Poisson extending structure of $(A,\ast_{1},\circ_{1})$ 
 through $V$ and the linear map 
 $\varphi:A\natural V\longrightarrow (E,\ast,\circ)$ with $\varphi(a,x)=a+x$
 is an isomorphism of pre-Poisson algebras whose restrictions on $A$ 
 and $V$ are the identity maps.
  	 \end{proof}
  	 
 By Theorem \ref{11}, any pre-Poisson algebra structure on a vector space $E$ containing $A$ 
 as a pre-Poisson subalgebra is isomorphic to a unified product of $A$ 
 through a given complement $V$ of $A$ in $E$. Hence the classification 
 of all pre-Poisson algebra structures on $E$ that contains $A$ as a 
 pre-Poisson subalgebra is equivalent to the classification of all 
 unified products $A\natural V$, associated to all pre-Poisson extending structures $\Omega(A,V)=(l_{1},r_{1},l_{2},r_{2},f,\ast_{2},\rho_1,\mu_1,\rho_2,\mu_2,g,\circ_{2})$,
  for a given complement $V$ of $A$ in $E$. 
  	
  	In the sequel, we will classify all unified products $A\natural V$.
  	 \begin{lem}\label{63}
  	 	Let $(A,\ast_{1},\circ_{1})$ be a pre-Poisson algebra and $V$ a vector space. Assume that $\Omega(A,V)=(l_{1},r_{1},l_{2},r_{2},f,\ast_{2},\rho_1,\mu_1,\rho_2,\mu_2,g,\circ_{2})$, $\Omega'(A,V)=(l'_{1},r'_{1},l'_{2},r'_{2},f',\ast'_{2},\rho'_1,\mu'_1,\rho'_2,\mu'_2,g'$,
  	 	$\circ'_{2})$
   are two pre-Poisson extending structures of $(A,\ast_{1},\circ_{1})$ through $V$, 
   and $A\natural V,~A \natural' V$ are their associated unified products respectively.
    Then there exists a bijection between the set of all homomorphisms of 
    pre-Poisson algebras $\psi: A\natural V\longrightarrow A \natural' V$ whose restriction on $A$ is 
    the identity map and the set of pairs $(\zeta,\eta)$, where $\zeta: V \longrightarrow A$,
     $\eta: V \longrightarrow V$ are linear maps and they satisfy the following conditions,
  	 	\begin{align}
  	 	&\eta(l_{1}(a)x)=l'_{1}(a)\eta(x),~~\eta(r_{1}(a)x)=r'_{1}(a)\eta(x),\label{90}\\
  	 	&\zeta(l_{1}(a)x)=a\ast_{1} \zeta(x)-r_{2}(x)a+r'_{2}(\eta(x))a,\label{91}\\
  	 	&\zeta(r_{1}(a)x)=\zeta(x)\ast_{1} a-l_{2}(x)a+l'_{2}(\eta(x))a,\label{92}\\
  	 	&\eta(x\ast_{2} y)=\eta(x)\ast'_{2}\eta(y)+l'_{1}(\zeta(x))\eta(y)+r'_{1}(\zeta(y))\eta(x),\label{93}\\
  	 	&\zeta(x\ast_{2}y)=\zeta(x)\ast_{1}\zeta(y)+l'_{2}(\eta(x))\zeta(y)+r'_{2}(\eta(y))\zeta(x)+f'(\eta(x),\eta(y))-f(x,y),\label{94}\\
  	 	&\eta(\rho_{1}(a)x)=\rho'_{1}(a)\eta(x),~~\eta(\mu_{1}(a)x)=\mu'_{1}(a)\eta(x),\label{95}\\
  	 	&\zeta(\rho_{1}(a)x)=a\circ_{1}\zeta(x)-\mu_{2}(x)a+\mu'_{2}(\eta(x))a,\label{96}\\
  	 	&\zeta(\mu_{1}(a)x)=\zeta(x)\circ_{1}a-\rho_{2}(x)a+\rho'_{2}(\eta(x))a,\label{97}\\
  	 	&\eta(x\circ_{2} y)=\eta(x)\circ'_{2}\eta(y)+\rho'_{1}(\zeta(x))\eta(y)+\mu'_{1}(\zeta(y))\eta(x),\label{98}\\
  	 	&\zeta(x\circ_{2}y)=\zeta(x)\circ_{1}\zeta(y)+\rho'_{2}(\eta(x))\zeta(y)+\mu'_{2}(\eta(y))\zeta(x)+g'(\eta(x),\eta(y))-g(x,y).\label{99}
  	 	\end{align}
  	 	 for any $a \in A$, $x, y \in V$.
 For a pair $(\zeta,\eta)$, the corresponding homomorphism of
pre-Poisson algebras $\psi=\psi_{(\zeta,\eta)}:A\natural V\longrightarrow A \natural' V$ is given by:
$\psi_{(\zeta,\eta)}(a,x)=(a+\zeta(x), \eta(x))$.
The homomorphism $\psi=\psi_{(\zeta,\eta)}$ is an isomorphism if and only if $\eta$ is a bijection
 and $\psi=\psi_{(\zeta,\eta)}$ co-stabilizes $V$ if and only if $\eta=id_V$. 
\begin{proof}
Due to $\psi:A\natural V \longrightarrow A \natural' V$ being an isomorphism of pre-Poisson algebras, 
  whose restriction on 
  $A$ is the identity map. Put $\psi(a,x)=(a+\zeta(x), \eta(x))$, for all $a\in A$ and $x \in V$, 
  where $\zeta: V \longrightarrow A$ and $\eta: V \longrightarrow V$ are two linear maps.
   On the basis of Lemma 5.6 \cite{4} and Lemma 1.9 \cite{5}, we can get the conclusion.
  	 	\end{proof}
  	 \end{lem}
  
\begin{defi}\label{01}
  	 Let $A$ be a pre-Poisson algebra and $V$ a vector space. Assume that
   $\Omega(A,V)=(l_{1},r_{1},l_{2},r_{2},f,\ast_{2},\rho_1,\mu_1,\rho_2,\mu_2,g,\circ_{2})$ and  $\Omega'(A,V)=(l'_{1},r'_{1},l'_{2},r'_{2},f',\ast'_{2},\rho'_1,\mu'_1,\rho'_2,\mu'_2,g',\circ'_{2})$ 
 are two pre-Poisson extending structures of $A$ through $V$. 
  If there is a pair $(\zeta,\eta)$ of linear maps, 
 where  $\eta:V\longrightarrow V$ is a bijection and $\zeta:V\longrightarrow A$, satisfying 
  Eqs. \eqref{90}-\eqref{99}, then
  $\Omega(A,V)$ and $\Omega'(A,V)$ are called {\bf equivalent} and denote it by 
  $\Omega(A, V) \equiv \Omega'(A, V)$.
  Moreover, if $\eta=id_{V}$, $\Omega(A, V)$ and $ \Omega'(A, V)$ are called {\bf cohomologous}, which is 
denoted by $\Omega(A, V) \approx \Omega'(A, V)$.
  	\end{defi}
  
Now, we are ready to provide an answer for the ES-problem of pre-Poisson algebra by the following Theorem.
  
  	\begin{thm}\label{13}
  		Let $A$ be a pre-Poisson algebra, $E$ a vector space containing $A$ as a subspace and $V$ a complement of $A$ in $E$. 
  		\begin{enumerate}
  			\item Denote  $H^2_{p}(V, A)= \mathcal{A}(A, V)/ \equiv$. The map
  		\begin{align*}
  			\Phi:H^2_{p}(V, A) \longrightarrow Extd(E,A), \quad \overline{\Omega(A, V)} \longmapsto \overline{A \natural  V}
  		\end{align*}
  		is a bijection, where $\overline{\Omega(A, V)}$ and $\overline{A \natural  V}$ 
  are the equivalent class of $\Omega(A, V)$ and $A \natural  V$ under $\equiv$ respectively. 
  	\item Denote $H^2(V, A)=\mathcal{A}(A, V) / \approx$. The map
  		\begin{align*} 
  			\Psi:H^2(V,A) \longrightarrow Extd'(E, A), \quad [\Omega(A, V)] \longmapsto [A \natural  V]
  		\end{align*}
  		is a bijection, where $[\Omega(A, V)]$ and $[A \natural  V]$ 
  are the cohomologous class of $\Omega(A, V)$ and $A \natural  V$ under $\approx$ respectively.

  \end{enumerate}
  	\end{thm}
  	\begin{proof}
The statements can be verified by Proposition \ref{5}, Theorem \ref{11} and  Lemma \ref{63} .
  	 	\end{proof}
  
\section{Special cases of unified products}

  In the section, we first introduce the notions of 
  crossed products and bicrossed products of pre-Poisson algebras, which
   are special cases of the unified products of pre-Poisson algebras.
  
  
\subsection{Crossed products and the extension problem}\label{12}

Let $(A,\ast_1,\circ_1)$ be a pre-Poisson algebra and $V$ a vector space. Suppose that
$\Omega(A,V)=(l_{1},r_{1},l_{2},r_{2},f,\ast_{2},\rho_1,\mu_1,\rho_2,\mu_2,g,\circ_{2})$ 
is a pre-Poisson extending datum of $(A,\ast_1,\circ_1)$ through $V$ 
such that $l_{1},r_{1},\rho_1,\mu_1$ are trivial maps, that is,
 $l_{1}(a)x=r_{1}(a)x=\rho_{1}(a)x=\mu_{1}(a)x=0$ for all $x \in V$ and $a \in A$.
 Then, $\Omega(A,V)=(l_{2},r_{2},f,\ast_{2},\rho_2,\mu_2,g,\circ_{2})$ is a pre-Poisson 
 extending structure of $A$ through $V$ if and only if $(V, \ast_{2}, \circ_{2})$ is a pre-Poisson algebra and they satisfy Eqs. \eqref{za5}-\eqref{za11}, \eqref{pra4}-\eqref{pra8}, \eqref{a48}-\eqref{a54} and \eqref{a60}-\eqref{a64}, where the pre-Poisson structure $(\ast,\circ)$ on $A\natural V$ is given by
 \begin{align}
	&(a,x)\ast(b,y)=(a\ast_{1}b+l_{2}(x)b+r_{2}(y)a+f(x,y),x\ast_{2}y),\label{sm1}\\
	&(a,x)\circ(b,y)=(a\circ_{1}b+\rho_{2}(x)b+\mu_{2}(y)a+g(x,y),x\circ_{2}y),~~\forall~~a,b \in A,x,y \in V.\label{sm2}
\end{align}
 In this case, $\Omega(A,V)=(l_{2},r_{2},f,\ast_{2},\rho_2,\mu_2,g,\circ_{2})$ is called
a {\bf crossed system of pre-Poisson algebra} of $(A,\ast_1,\circ_1)$ through $V$, denoted by $CS(A,V)$,
and the associated unified product $A\natural V$ is called the {\bf crossed product} 
of the pre-Poisson algebras $A$ and $V$, which is denoted by $A\sharp V$.
 Denote by $GCS(A,V)$ the set of all crossed systems of the pre-Poisson algebra $A$ through $V$.

Assume that $(l_{2},r_{2},f,\rho_{2},\mu_{2},g)$ is a crossed system 
of the pre-Poisson algebra $A$ through $V$. 
Then, $A \cong A \times \{ 0\}$ is an ideal in the pre-Poisson algebra 
 $A\sharp V$ since $(a,0)\ast(b,y)=(a\ast_{1}b+r_{2}(y)a,0)$ and $(a,0)\circ(b,y)=(a\circ_{1}b+\mu_2 (y)a,0)$.
 On the other hand, crossed products describe all pre-Poisson algebra structures on a vector space $E$ 
 such that a given $A$ is an ideal of $E$.

\begin{pro}\label{62}
Suppose that $A$ is a pre-Poisson algebra and $E$ is a vector space containing $A$ as a subspace. 
Then any pre-Poisson algebra structure on $E$ that contains $A$ as a
 pre-Poisson ideal is isomorphic to a crossed product $A\sharp V$. Moreover,
 the isomorphism of pre-Poisson algebras 
  $(E,\ast,\circ) \cong A\sharp V$ can be chosen such
  that it stabilizes $A$ and co-stabilizes $V$.
\end{pro}

\begin{proof}
	 Assume that $E$ is a pre-Poisson algebra such that $A$ is a pre-Poisson ideal of $E$.
 Then $A$ is a pre-Poisson subalgebra of $E$ 
and there is a linear map $p:E\longrightarrow A$ such that $p(a)=a$ for all $a\in A$. 
In view of Theorem \ref{11}, there is a pre-Poisson extending structure 
$\Omega(A, V)=(l_{1},r_{1},l_{2},r_{2},f,\ast_{2},\rho_1,\mu_1,\rho_2,\mu_2,g,\circ_{2})$ of $A$
 through the subspace $V=\mathrm{ker}(p)$ of $E$ and an isomorphism of pre-Poisson algebras $E \cong A\natural V$. 
Since $A$ is a pre-Poisson ideal of $E$, for any $x \in V$ and $a \in A$, 
we have that $p(a\ast x)=a\ast x,~p(x\ast a)=x\ast a,~p(a\circ x)=a\circ x,~p(x\circ a)=x\circ a$, 
which indicate that 
$l_{1}(a)x=a\ast x-p(a\ast x)=0,~r_{1}(a)x=x\ast a-p(x\ast a)=0,~\rho_{1}(a)x=a\circ x-p(a\circ x)=0,~\mu_{1}(b)x=x\circ a-p(x\circ a)=0$ 
and $E$ is isomorphic to the pre-Poisson algebra associated to 
the crossed system $(A,V, l_{2},r_{2},f,\rho_2,\mu_2,g)$.
\end{proof}

In the light of Proposition \ref{62}, the answer to the restricted version of the ES-problem
reduces to the classification of all crossed products associated to all crossed systems. 
This is what we do next by explicitly constructing a classification
object, denoted by $GH^2( V,A )$. 

 Combining Proposition \ref{5}, Lemma \ref{63}, Definition \ref{01} and Proposition \ref{62}, we have  
 
\begin{thm} \label{100}
Let $(A,\ast_{1},\circ_{1})$ be a pre-Poisson algebra and $E$ a vector space containing $A$ as a subspace. 
Then $\approx$ is an equivalent relation on the set $GCS(A,V)$ of all crossed systems of $A$ through by 
a subspace $V$ of $E$. Denote $GH^{2}(V,A)=GCS(A,V)/\approx$, then the map
$\Theta:GH^{2}(V,A) \longrightarrow  Extd'(E, A)$ is a bijection, where $\Theta$ assigns 
the equivalent class of $CS(A, V)$ to the equivalent class of $A\sharp V$. 
\end{thm}

Calculating $GH^{2}(V,A)$ is not a trivial problem. We connect it with the classical extension problem.

\begin{defi}
 Let $A$ and $B$ be two pre-Poisson algebras.
 A {\bf non-abelian extension} of $A$ by $B$ is a pre-Poisson algebra $E$ equipped with a short exact sequence of pre-Poisson algebras:
	\begin{align*}
		\mathcal{E}:0\longrightarrow B\stackrel{i}{\longrightarrow} E\stackrel{p}{\longrightarrow}A\longrightarrow0.
	\end{align*}
When $B$ is abelian, the extension is called an {\bf abelian extension} of $A$ by $B$. A {\bf section} of $p$ is a linear map $s: A
\rightarrow E $ such that $ps = I_{A}$.
\end{defi}

Let $ E_1$ and $E_2$ be two non-abelian extensions of $A$ by $B$. They are said to be
{\bf cohomologous} if there is an isomorphism of pre-Poisson algebras
$\varphi:E_1\longrightarrow E_2$ such that
the following commutative diagram holds:
 \begin{equation}\label{E1} \xymatrix{
  0 \ar[r] & B\ar@{=}[d] \ar[r]^{i_1} & E_1\ar[d]_{\varphi} \ar[r]^{p_1} & A \ar@{=}[d] \ar[r] & 0\\
 0 \ar[r] & B \ar[r]^{i_2} & E_2 \ar[r]^{p_2} & A  \ar[r] & 0
 .}\end{equation}
 
Obviously, any crossed product $A\sharp V$ is an extension of the pre-Poisson algebra $V$ by $A$ 
equipped with the short exact sequence of pre-Poisson algebras:
\begin{align}\label{E2}
	0\longrightarrow A\stackrel{i_A}{\longrightarrow} A\sharp V \stackrel{\pi_V}{\longrightarrow}V\longrightarrow0,
	\end{align}
where $i_A$ and $\pi_V$ are injection and projection respectively. 

It is natural to describe all pre-Poisson algebras $E$ which are extensions 
of $A$ by $B$. This is the classical extension problem. 

\begin{thm} \label{Es}
Let $(A,\ast_{1},\circ_{1})$ and $(V,\ast_{2},\circ_{2})$ be two pre-Poisson algebras.
Assume that
 \begin{equation*}0\longrightarrow A\stackrel{i}{\longrightarrow} E\stackrel{p}{\longrightarrow} V\longrightarrow0\end{equation*} 
  is
an extension of $(V,\ast_{2},\circ_{2})$ by
$(A,\ast_{1},\circ_{1})$ with a section $s$ of $p$. Then any 
extension of $V$ by $A$ is cohomologous to a crossed extension (\ref{E2}). 
\end{thm}

\begin{proof}Define linear maps
\begin{align*}l_{2},r_{2},\rho_{2},\mu_{2} : V \longrightarrow End(A)\end{align*}
  and bilinear maps 
  \begin{align*} f,g: V\times V \longrightarrow A\end{align*}
  respectively by
 \begin{equation}\label{sm3} l_{2}(x)a=s(x)\ast_{E} a,~~r_{2}(x)a=a\ast_{E} s(x),\end{equation} 
  \begin{equation}\label{sm4} \rho_{2}(x)a=s(x)\circ_{E} a,~~\mu_{2}(x)a=a\circ_{E} s(x),\end{equation} 
  \begin{equation}\label{sm5} f(x,y)=s(x)\ast_{E} s(y)-s(x\ast_2 y),~~g(x,y)=s(x)\circ_{E} s(y)-s(x\circ_2 y),\end{equation} 
  for all $x,y\in V$ and $a\in A$. Then  
  \begin{equation}\varphi:  A\oplus V \to E,~\varphi(a,x)=a+s(x),~\forall~a\in A,~x\in V\end{equation}  
  is a linear isomorphism with the inverse  
  \begin{equation}\varphi^{-1}:E\longrightarrow A\oplus V,~\varphi^{-1}(e)=(e-sp(e),p(e)),~\forall~e\in E.\end{equation} 
   We only need to make sure that the pre-Poisson structure $(\ast,\circ)$ that can be defined on $A\oplus V$ such that 
  $\varphi$ is an isomorphism of pre-Poisson algebras coincides with the one given by Eqs.\eqref{sm1}-\eqref{sm2}.
   In fact, since $p$ is homomorphism of pre-Poisson algebras, for all $a,b\in A$ and $x,y\in V$, we have
 \begin{align}\label{sm8}&p(a\ast_{E} b+s(x)\ast_{E}s(y)+s(x)\ast_{E} b+a\ast_{E} s(y))
\\  \notag=& p(a)\ast_{E} p(b)+ps(x)\ast_{E}ps(y)+ps(x)\ast_{E} p(b)+p(a)\ast_{E} ps(y)
 \\  \notag=&x\ast_{E}y=x\ast_{2}y
 \end{align}
 and
  \begin{equation}\label{sm9}sp(a\ast_{E} b+s(x)\ast_{E}s(y)+s(x)\ast_{E} b+a\ast_{E} s(y))=s(x\ast_{2}y).\end{equation}
 By Eqs. \eqref{sm3}-\eqref{sm9}, we obtain
 \begin{align*}
 	&(a,x)\ast (b,y)=\varphi^{-1}(\varphi(a,x)\ast_{E} \varphi(b,y))
 \\=&\varphi^{-1}((a+s(x))\ast_{E} (b+s(y)))\\=&\varphi^{-1}(a\ast_{E} b+s(x)\ast_{E}s(y)+s(x)\ast_{E} b+a\ast_{E} s(y))
 \\=&(a\ast_1 b+l_{2}(x)b+r_{2}(y)a+f(x,y),x\ast_2 y).
 \end{align*}
 Analogously, 
 \begin{equation*}(a,x)\circ(b,y)=(a\circ_1 b+\rho_{2}(x)b+\mu_{2}(y)a+g(x,y),x\circ_2 y).\end{equation*}
In addition, it is easy to check that the following diagram is commutative:
    \begin{equation*} \xymatrix{
  0 \ar[r] & A\ar@{=}[d] \ar[r]^{i_A} & A\oplus V\ar[d]_{\varphi} \ar[r]^{\pi_V} & V \ar@{=}[d] \ar[r] & 0\\
 0 \ar[r] & A \ar[r]^{i} & E \ar[r]^{p} & V  \ar[r] & 0 .}\end{equation*}
  The proof is completed. 
\end{proof}

Theorem \ref{Es} indicates that computing the classifying object $ Extd'(E, A)$ reduces to the classification
of all crossed extension of (\ref{E2}).

Let $CS(A,V)=(l_{2},r_{2},f,\ast_{2},\rho_{2},\mu_{2},g,\circ_{2})$ be a crossed system of
$(A,\ast_{1},\circ_{1})$ through vector space $V$. For a fixed pre-Poisson algebra $(V,\ast_{2},\circ_{2})$,
let $CS_{l} (A,V)$ be the set of six maps $l_{2},r_{2},f,\rho_{2},\mu_{2},g$
 satisfying Eqs. \eqref{za5}-\eqref{za11}, \eqref{pra4}-\eqref{pra8}, \eqref{a48}-\eqref{a54} and \eqref{a60}-\eqref{a64},
  which is called the 
{\bf local crossed system (non-abelian 2-cocycle)} of $(V,\ast_{2},\circ_{2})$ by $(A,\ast_{1},\circ_{1})$.

\begin{defi}Two local crossed systems $CS_{l} (A,V)=(l_{2},r_{2},f,\rho_{2},\mu_{2},g)$ and 
$CS_{l} ^{'}(A,V)=(l'_{2},r'_{2},f',\rho'_{2},\mu'_{2},g')$ are called {\bf cohomologous}, denoted by 
$CS_{l} (A,V)\approx CS_{l}^{'}$
$ (A,V)$,
if there is a linear map $\zeta:V \longrightarrow A$ satisfying 
Eqs.~(\ref{91})-(\ref{94}), (\ref{96})- (\ref{97}) and (\ref{99}).\end{defi}

Denote by $Z^{2}_{nab}(V,A)$ the set of all local crossed systems $CS_{l} (A,V)$ of $(V,\ast_{2},\circ_{2})$ 
through $(A,\ast_{1},\circ_{1})$. The quotient of $Z^{2}_{nab}(V,A)$ by the above
equivalence relation is denoted by $H^{2}_{nab}(V,A)$, which is called the non-abelian second cohomology group
of $V$ in $A$. It is easy to check that $H_{nab}^{2}(V,A)=H^{2}(V,A)$ when
$A$ is an abelian pre-Poisson algebra, where $H^{2}(V,A)$ is the second cohomology group
of $V$ in $A$. 

Analogous to Theorem \ref{100}, we have

\begin{thm} \label{101}
Let $(A,\ast_{1},\circ_{1})$ be a pre-Poisson algebra.
\begin{enumerate}
	\item If $(V,\ast_{2},\circ_{2})$ is a fixed pre-Poisson algebra. The map
	\begin{equation*}\Phi:H^{2}_{nab}(V,A)\longrightarrow  Extd'(E, A),~~\Phi([Z^{2}_{nab}(V,A)])=[A\sharp V]\end{equation*}
	is a bijection, where $[Z^{2}_{nab}(V,A)]$ is the equivalent class of $Z^{2}_{nab}(V,A)$
	and $[A\sharp V]$ is the equivalent class of the crossed extension defined by (\ref{E2}).
	\item Let $E$ a vector space containing $A$ as a pre-Poisson ideal and $V$ a vector subspace of $E$.
	Then 
	\begin{equation*} GH^{2}(V,A)=\bigsqcup_{(V,\ast_{2},\circ_{2})}H^{2}_{nab}(V,A).\end{equation*}
\end{enumerate}
\end{thm}

In general, it is not easy to calculate $GH^{2}(V,A)$. When $A$ and $V$ are abelian pre-Poisson algebras, we denote by $A_0$ and $V_{0}$ respectively.
Then $GH^{2}(A_0,V_0)=H^{2}_{nab}(V_0,A_0)$.

\begin{ex}
Let $k_0^{n}$ be an abelian pre-Poisson algebra of dimension $n$ with basis $\{E_1,E_2,\cdot\cdot\cdot,E_n \}$
and
$k_0$ a 1-dimensional abelian pre-Poisson algebra with basis $E_{n+1}$. Suppose that
 $CS_{l} (k_0^{n},k_0)=(\rho_{2},\mu_{2},g,l_{2},r_{2},f)$ 
is the local system of $k_0^{n}$ through $k_0$ and 
$(k_0^{n}\sharp k_0,\ast,\circ)$ is the corresponding 
crossed product of $k_0^{n}$ through $k_0$. Assume that
\begin{align*}
		&E_{n+1}\circ E_{i}:=\rho_2(E_{n+1})E_i=\sum_{j = 1}^n a_{ji}E_{j},\  \ \ E_{i}\circ E_{n+1}:=\mu_2(E_{n+1})E_i=\sum_{j = 1}^n b_{ji}E_{j},\\
&E_{n+1}\circ E_{n+1}:=g(E_{n+1},E_{n+1})=\sum_{j = 1}^n \theta_{0j}E_{j}, \  \ \
		E_{n+1}\ast E_{i}:=l_2(E_{n+1})E_i=\sum_{j = 1}^n c_{ji}E_{j},
\\&E_{i}\ast E_{n+1}:=r_2(E_{n+1})E_i=\sum_{j = 1}^n d_{ji}E_{j}, \ \ \ E_{n+1}\ast E_{n+1}:=f(E_{n+1},E_{n+1})=\sum_{j = 1}^n \upsilon_{0j}E_{j}.
	\end{align*}	
Combining Eqs.~(\ref{a40})-(\ref{a65}), we have
\begin{align*}
	&AB=BA+B^{2}, \ \ DC=CD+D^{2}, \ \ C^{2}=0, \ \ AC=CA=0, \\&BC=CB+DB=BD-AD=-DA, \ \
	A\upsilon_{0}=B\upsilon_{0}=C\theta_{0},~C\upsilon_{0}=2D\upsilon_{0},
\end{align*}
where $A=(a_{ij})_{n\times n},~B=(b_{ij})_{n\times n},~C=(c_{ij})_{n\times n},~D=(d_{ij})_{n\times n}\in M_{n}(k)$ and
$\theta_{0}=(\theta_{0j})_{n\times 1},\upsilon_{0}=(\upsilon_{0j})_{n\times 1}\in k^{n}$ are matrices over the field $k$.

Two six-tuples $(A,B,C,D,\upsilon_{0},\theta_{0})$ and $(A',B',C',D',\upsilon'_{0},\theta'_{0})$ are cohomologous,
denoted by $(A,B,C,D,\upsilon_{0},\theta_{0})\approx(A',B',C',D',\upsilon'_{0},\theta'_{0})$,
if and only if $A=A'$, $B=B'$, $C=C'$, $D=D'$ and there exists $w\in k^{n}$ 
such that $\theta_{0}-\theta'_{0}=(A+B)w,~\upsilon_{0}-\upsilon'_{0}=(C+D)w$.
Let $\mathfrak{C}(n)$ be the set of all six-tuples $(A,B,C,D,\upsilon_{0},\theta_{0})$. 
Then, $GH^{2}(k_{0},k^{n}_{0})\cong \mathfrak{C}(n)/\approx$. In the following, we give a detail.

Assume that $Z^{2}_{nab}(k_{0},k^{n}_{0})$ is the set of all local systems of
 $k^{n}_{0}$ through $k_{0}$. 
 Denote by $L$ the set of all six-tuples $(\alpha,\beta,\theta_{0},\gamma,\lambda,\upsilon_{0})$,
 where $\alpha,\beta,\gamma,\lambda:k^{n}\longrightarrow k^{n}$ are three linear maps 
and 
 $\theta_{0},\upsilon_{0}\in k^{n}$ are vectors, satisfying the compatibility conditions 
\begin{align*}
		&\alpha \beta=\beta \alpha+\beta^{2}, \ \ \lambda \gamma=\gamma \lambda+\lambda^{2},\ \ \gamma^{2}=0,\ \
		\alpha \gamma=\gamma \alpha=0,\\&\beta \gamma=\gamma\beta+\lambda \beta=\beta\lambda-\alpha \lambda=-\lambda\alpha,\ \
		\alpha \upsilon_{0}=\beta \upsilon_{0}=\gamma \theta_{0}, \ \ \gamma \upsilon_{0}=2\lambda \upsilon_{0}.
	\end{align*} 
It is easy to check that there is a bijection between the sets $Z^{2}_{nab}(k_{0},k^{n}_{0})$ and $L$.
Thus, we can obtain a local system $(\rho_2,\mu_2,g,l_{2},r_{2},f)$ corresponding to 
$(\alpha,\beta,\theta_{0},\gamma,\lambda,\upsilon_{0})$ as follows:
	\begin{align*}
		&\rho_2(x) a:=a\alpha(x), \ \  \ \mu_2(x)a:=a\beta(x), \ \ \ g(x,y):=xy\theta_{0},\\
		&l_{2}(x)a:=a\gamma(x), \ \  \ r_{2}(x)a:=a\lambda(x), \ \ \ f(x,y):=xy\upsilon_{0},
	\end{align*} 
	for all $x,y\in k^{n}$ and $a\in k$. Assume that $\{e_{1},e_{2},\cdots,e_{n}\}$ is the canonical basis of $k^{n}$.
Let $A,B,C,D$ be the matrices associated to $\alpha,\beta,\gamma$ and $\lambda$ with respect to this basis respectively.
 If we take the vectors $E_{1}=(e_{1},0),\cdots,E_{n}=(e_{n},0)$ and $E_{n+1}=(0,1)$ as a basis in $k^{n}\times k$, then the pre-Poisson algebra $k^{n+1}_{(A,B,C,D,\upsilon_{0},\theta_{0})}$ is just the crossed product $k^{n}_{0}\sharp k_{0}$ with the structures given by 
  \begin{align*}
		&(a,x)\ast (b,y)=(l_2(x)b+r_2(y)a+f(x,y),0),\\
		&(a,x)\circ (b,y)=(\rho_2(x)b+\mu_2(y)a+g(x,y),0),
	\end{align*}
  for all $x,y\in k_{0}$ and $a,b\in k^{n}_{0}$. Therefore, $GH^{2}(k_{0},k^{n}_{0})\cong \mathfrak{C}(n)/\approx$.
\end{ex}

\subsection{The factorization problem}

Let $\Omega(A,V)=(l_{1},r_{1},l_{2},r_{2},f,\ast_{2},\rho_1,\mu_1,\rho_2,\mu_2,g,\circ_{2})$ 
    be an extending datum of $A$ through $V$. If $f$ and $g$ are trivial, by Theorem \ref{11}, we get that $\Omega(A,V)$
 is an extending structure of $A$ through $V$ if and only if $(V,\ast_2,\circ_2)$ is a pre-Poisson algebra,
$(A, l_2,r_2,\rho_{2},\mu_{2}) $ is a bimodule of $V$, $(V,\rho_{1},\mu_{1},l_{1},r_{1})$ is a bimodule of $A$
and they satisfy Eqs. (\ref{za2})-(\ref{za7}), (\ref{pra2})-(\ref{pra5}), (\ref{a45})-(\ref{a50}) and (\ref{a58})-(\ref{a61}). In this case, the associated
unified product
$A\natural V$ is denoted by $A \bowtie V$, called the bicrossed product of $A$ by $V$ and $(A,V,l_{1},r_{1},l_{2},r_{2},\rho_1,\mu_1,\rho_2,\mu_2)$ is
called the matched pair of pre-Poisson algebras $A$ and $V$. 
For the completeness, we write them in detail.

\begin{thm}\label{thm:3.8}
	 Let $(A_1,\ast_1,\circ_1)$ and $(A_2,\ast_2,\circ_2)$ be two pre-Poisson
algebras. Suppose that there are linear maps
$l_{1},r_{1},\rho_1,\mu_1:A_1\longrightarrow
\hbox{End} (A_2)$ and
$l_{2},r_{2},\rho_2,\mu_2:A_2\longrightarrow
\hbox{End} (A_1)$.
Define multiplications on $A_1\oplus A_2$   by
\begin{eqnarray}
&&(x,a)\ast
(y,b)=(x{\ast_1}y+l_{2}(a)y+r_{2}(b)x,a{\ast_2}b+l_{1}(x)b+r_{1}(y)a)
,\label{ppmp eq1.1}\\
&&(x,a)\circ
(y,b)=(x{\circ_1}y+\rho_2(a)y+\mu_2(b)x,a{\circ_2}b+\rho_1(x)b+\mu_1(y)a)
,~~\forall~x,y\in A_1,a,b\in A_2.\label{ppmp eq1.2}
\end{eqnarray}
Then $(A_{1}\oplus A_{2},\ast,\circ)$ is a pre-Poisson algebra if and only if  the following conditions hold:
\begin{enumerate}
	\item $(A_2,l_{1},r_{1},\rho_1,\mu_1)$ is a
	representation of $(A_1,\ast_1,\circ_1)$.
	\item $(A_1,l_{2},r_{2},\rho_2,\mu_2)$ is a
	representation of $(A_2,\ast_2,\circ_2)$.
	\item $(A_1,
	A_2, l_{1},r_{1},l_{2},r_{2})$ is a matched pair
	of Zinibel algebras.
	\item $(A_1,
	A_2,\rho_1,\mu_1,\rho_2,\mu_2)$ is a matched
	pair of pre-Lie algebras.
	\item The
	following compatible conditions hold:
	\begin{align}&\label{ppmp eq1.3}r_{2}(a)(x\circ_{1} y-y\circ_{1} x)=x\circ_{1}
		r_{2}(a)y-y\ast_{1}\mu_2(a)x+\mu_2(l_{1}(y)a)x-r_{2}(\rho_1(x)a)y,\\&
\label{ppmp eq1.4}(\mu_2(a)x-\rho_2(a)x)
		\ast_{1} y+l_{2}(\rho_1(x)a-\mu_1(x)a)y=x\circ_{1}
		l_{2}(a)y+\mu_2(r_{1}(y)a)x-
		l_{2}(a)(x\circ_1 y),\\&\label{ppmp eq1.5}(\rho_2(a)x-\mu_2(a)x)
		\ast_{1} y+l_{2}(\mu_1(x)a-\rho_1(x)a)y=
		\rho_2(a)(x\ast_{1}y)-x\ast_{1}
		\rho_2(a)y-r_{2}(\mu_1(y)a)x,
\\&\label{ppmp eq1.6}\mu_2(a)(x\ast_{1} y+y\ast_{1} x)=x\ast_{1}
		\mu_2(a)y+y\ast_{1}
		\mu_2(a)x+r_{2}(\rho_1(y)a)x+r_{2}(\rho_1(x)a)y,
\\&\label{ppmp eq1.7}(r_{2}(a)x+l_{2}(a)x)
		\circ_{1} y+\rho_2(l_{1}(x)a+r_{1}(x)a)y=x\ast_{1}
		\rho_2(a)y+r_{2}(\mu_1(y)a)x+
		l_{2}(a)(x\circ_1 y),\\&\label{ppmp eq1.8}r_{1}(x)(a\circ_{2} b-b\circ_{2} a)=a\circ_{2}
		r_{1}(x)b-b\ast_{2}
		\mu_1(x)a+\mu_1(l_{2}(b)x)a-r_{1}(\rho_2(a)x)b,\\&\label{ppmp eq1.9}(\mu_1(x)a-\rho_1(x)a)
		\ast_{2} b+l_{1}(\rho_2(a)x-\mu_2(a)x)b=a\circ_{2}
		l_{1}(x)b+\mu_1(r_{2}(b)x)a-
		l_{1}(x)(a\circ_2 b),\\&\label{ppmp eq1.10}(\rho_1(x)a-\mu_1(x)a)
		\ast_{2} b+l_{1}(\mu_2(a)x-\rho_2(a)x)b=
		\rho_1(x)(a\ast_{2}b)-a\ast_{2}
		\rho_1(x)b-r_{1}(\mu_2(b)x)a,\\&\label{ppmp eq1.11}\mu_1(x)(a\ast_{2} b+b\ast_{2} a)=a\ast_{2}
		\mu_1(x)b+b\ast_{2}
		\mu_1(x)a+r_{1}(\rho_2(b)x)a+r_{1}(\rho_2(a)x)b,\\&\label{ppmp eq1.12}(r_{1}(x)a+l_{1}(x)a)
		\circ_{2} b+\rho_1(l_{2}(a)x+r_{2}(a)x)b=a\ast_{2}
		\rho_1(x)b+r_{1}(\mu_2(b)x)a+
		l_{1}(x)(a\circ_2 b),\end{align}
	for all $x,y\in A_1,a,b\in A_2$.
\end{enumerate}
In this case, we
denote this pre-Poisson algebra by $A_1\bowtie A_2$, and 
$(A_1, A_2,
l_{1},r_{1},\rho_1,\mu_1,l_{2},r_{2},\rho_2,\mu_2)$
satisfying the above conditions is called a {\bf matched pair of
pre-Poisson algebras}. \end{thm}
\begin{proof} It is a special case of Theorem \ref{5}.
\end{proof}

{\bf Factorization problem}: Let $A$ and $B$ be two pre-Poisson algebras. Characterize and 
classify all pre-Poisson algebras $C$ that factorize through $A$ and $B$, 
i.e. $C$ contains $A$ and $B$ as pre-Poisson subalgebras such that $C=A+B$ and $A \cap B=\{0\}$.

In the light of Theorem \ref{11}, we easily get the following statement.
\begin{cor}A pre-Poisson algebra $C$ factorizes through two pre-Poisson algebras $A$ and $B$ if and only if there is a matched pair of pre-Poisson algebras $(A,B, l_{1},r_{1},l_{2},r_{2},\rho_1,\mu_1,\rho_2,\mu_2)$ such that $C \cong A \bowtie B$.
\end{cor}
\subsection{Flag extending structures}
In this section, we investigate the unified product $A\natural V$ when $dim(V ) = 1$.

\begin{defi}
	Let $(A,\ast_{1})$ be a Zinbiel algebra. A \textbf{flag datum} of $(A,\ast_{1})$ is a 6-tuple $(a_{1},k_{1},\tau,\omega$,
	$P,Q)$, where $a_{1}\in A,k_{1}\in k$, and $\tau,\omega:A \longrightarrow k$, $P,Q:A \longrightarrow A$ are linear maps satisfying the following compatibility conditions for $a,b\in A$,
	\begin{align*}
		&\tau(a\ast_{1} b)=\tau(a)\tau(b)-\tau(b\ast_{1} a), \ \ \ 
		\omega(a\ast_{1} b)=\tau(a)\omega(b),\\
		&\omega(a)\omega(b)=\omega(Q(a))=k_1\omega(a)=\tau(P(a)+Q(a))=0,\ \ \
		 P(a\ast_{1} b)=a\ast_{1}P(b)+\omega(b)Q(a),\\
		&P(a\ast_{1} b)=(P(a)+Q(a))\ast_{1}b+(\tau(a)+\omega(a))P(b), \ \ \ Q(a\ast_{1} b)=a\ast_{1}Q(b)+\tau(b)Q(a)-Q(b\ast_{1} a),\\
		&P^{2}(a)=2a_{1}\ast_{1}a+2k_{1}P(a)-\omega(a)a_{1},\ \ \
		Q^{2}(a)=P(Q(a))-Q(P(a))-\omega(a)a_{1},\\
		&P(Q(a))=a\ast_{1}a_{1}+k_{1}Q(a)-\tau(a)a_{1}, \ \ \ 
		P(a_{1})=2Q(a_{1})+k_{1}a_{1},\ \ \ 
		\omega(a_{1})=2\tau(a_{1})+k^{2}_{1}.
	\end{align*}
Denote by $\mathcal{FN}(A)$ the set of all flag datums of $(A,\ast_{1})$.
\end{defi}
\begin{pro}\label{71}
	Let $(A,\ast_{1})$ be a Zinbiel algebra and $V$ a vector space of dimension 1 with a basis $\{x\}$. 
Then there exists a bijection between the set $\mathcal{A}(A,V)$ of all extending structures of $A$ through $V$ 
and the set of all flag datums $\mathcal{FN}(A)$. 
The Zinbiel extending structure $\Omega(A,V)=(l_{1},r_{1},l_{2},r_{2},f,\ast_{2})$ corresponding to $(a_{1},k_{1},\tau,\omega,P,Q)\in\mathcal{FN}(A)$ is given by
	\begin{align*}
		&l_{1}(a)x=\tau(a)x,~r_{1}(a)x=\omega(a)x,~~l_{2}(x)a=P(a),~r_{2}(x)a=Q(a),~~f(x,x)=a_{1},~x\ast_{2}x=k_{1}x.
	\end{align*}
	The associated unified product $A\natural V$ is given by
	\begin{align*}
		(a,x)\ast(b,x):=(a\ast_{1}b+Q(a)+P(b)+a_{1},\tau(a)x+\omega(b)x+k_{1}x)
	\end{align*}
	and the Zinbiel algebra $A\natural {x}$ is generated by the following relations
	\begin{align*}
		x\ast x=a_{1}+k_{1}x,~a\ast x=Q(a)+\tau(a)x,~x\ast a=P(a)+\omega(a
		)x.
	\end{align*}
\end{pro}
The special case of the flag datum for Zinbiel algebras was studied in \cite{5}, where $\tau=0$.

\begin{defi}\cite{6}
	Let $(A,\circ_{1})$ be a pre-Lie algebra. A \textbf{flag datum} of $(A,\circ_{1})$ is a 6-tuple $(a_{2},k_{2},p,q,S,T)$, 
where $a_{2}\in A,k_{2}\in k$, and $p,q:A \longrightarrow k$, $S,T:A \longrightarrow A$ 
are linear maps satisfying the following compatibility conditions for $a,b\in A$,
	\begin{align*}
		&p(a\circ_{1} b)=p(b\circ_{1} a),\ \ \ \
		q(a\circ_{1}b)=q(a)q(b),\\
		&S(a\circ_{1} b)=S(a)\circ_{1}b+a\circ_{1}S(b)+(q(a)-p(a))S(b)+q(b)T(a)-T(a)\circ_{1}b,\\
		&T(a\circ_{1}b)-T(b\circ_{1}a)=p(b)T(a)-p(a)T(b)+a\circ_{1}T(b)-b\circ_{1}T(a),\\
		&T^{2}(a)=T(S(a))-S(T(a))+a\circ_{1}a_{2}+(q(a)-2p(a))a_{2}+k_{2} T(a),\\
		&p(S(a))-p(T(a))=q(T(a))+k_{2}(p(a)-q(a)).
	\end{align*}
	Denote by $\mathcal{FL}(A)$ the set of all flag datums of $(A,\circ_{1})$.
\end{defi}
\begin{pro}\cite{6}\label{70}
	Let $(A,\circ_{1})$ be a pre-Lie algebra and $V$ a vector space of dimension 1 with a basis $\{x\}$. 
Then there exists a bijection between the set $\mathcal{A}(A,V)$ of all extending structures of $A$ through $V$ 
and the set of all flag datums $\mathcal{FL}(A)$. The pre-Lie extending structure 
$\Omega(A,V)=(\rho_1,\mu_1,\rho_2,\mu_2,g,\circ_{2})$ corresponding to $(a_{2},k_{2},p,q,S,T)\in\mathcal{FL}(A)$ is given by
	\begin{align*}
		\rho_1(a)x=p(a)x,~\mu_1(a)x=q(a)x,~~\rho_2(x)a=S(a),~\mu_2(x)a=T(a),~~g(x,x)=a_{2},~x\circ_{2} x=k_{2} x.
	\end{align*}
	The associated unified product $A\natural V$ is given by
	\begin{align*}
		(a,x)\circ(b,x):=(a\circ_{1}b+T(a)+S(b)+a_{2},p(a)x+q(b)x+k_{2} x)
	\end{align*}
	and the pre-Lie algebra $A\natural {x}$ is generated by the following relations
	\begin{align*}
		x\circ x=a_{2}+k_{2} x,~a\circ x=T(a)+p(a)x,~x\circ a=S(a)+q(a)x.
	\end{align*}
\end{pro}
\begin{defi}
	Let $(A,\ast_{1},\circ_{1})$ be a pre-Poisson algebra. 
A \textbf{flag datum} of $(A,\ast_{1},\circ_{1})$ is a 12-tuple
	 $(a_{1},k_{1},\tau,\omega,P,Q,a_{2},k_{2},p,q,S,T)$, where $a_{1},a_{2}\in A$, $k_{1},k_{2}\in k$, 
and $\tau,\omega,p,q:A \longrightarrow k$, $P,Q,S,T:A \longrightarrow A$ 
are linear maps satisfying the following compatibility conditions for $a,b\in A$,
	 \begin{align}
	 	&\text{$(a_{1},k_{1},\tau,\omega,P,Q)$ is a flag datum of the Zinbiel algebra $(A,\ast_{1})$,}\label{73}\\
	 	&\text{$(a_{2},k_{2},p,q,S,T)$is a flag datum of the pre-Lie algebra $(A,\circ_{1})$,}\\
	 	&\tau(a\circ_{1}b)=\tau(b\circ_{1}a),\\
	 	&\omega(b)q(a)=\omega(a\circ_{1}b),\\
	 	&\omega(b)(q(a)-p(a))=q(a\ast_{1}b)-\tau(a)q(b),\\
	 	&(q(a)-p(a))k_{1}+\tau(S(a)-T(a))=\omega(T(a))=q(Q(a)),\\
	 	&\omega(a)k_{2}+q(P(a))-q(a)k_{1}=0,\\
	 	&Q(a\circ_{1} b-b\circ_{1} a)=a\circ_{1}Q(b)+\tau(b)T(a)-b\ast_{1}T(a)-p(a)Q(b),\\
	 	&(T(a)-S(a))\ast_{1}b+(p(a)-q(a))P(b)=a\circ_{1}P(b)+\omega(b)T(a)-P(a\circ_{1}b),\\
	 	&(S(a)-T(a))\ast_{1}b+(q(a)-p(a))P(b)=S(a\ast_{1}b)-a\ast_{1}S(b)-q(b)Q(a),\\
	 	&Q(T(a)-S(a))+a_{1}(2p(a)-q(a))=a\circ_{1}a_{1}+k_{1} T(a)-P(T(a)),\\
	 	&Q(S(a)-T(a))+a_{1}(q(a)-p(a))=S(Q(a))+a_{2}\tau(a)-a\ast_{1}a_{2}-k_{2} Q(a),\\
	 	&S(P(a))-P(S(a))+a_{2}\omega(a)-a_{1}q(a)=0,\\
	 	&S(a_{1})+k_{1}a_{2}-P(a_{2})-k_{2} a_{1}=0,\\
	 	&p(a_{1})=q(a_{1})=\omega(a_{2}),\\
	 	&p(a\ast_{1}b+b\ast_{1}a)=\tau(a)p(b)+\tau(b)p(a),\\
	 	&q(b)(\tau(a)+\omega(a))=\tau(a)q(b)+\omega(a\circ_{1}b),\\
	 	&k_{2}\omega(a)+p(P(a)+Q(a))=k_{1}p(a)+\omega(\tau(a)),\\
	 	&\omega(S(a))=0,\\
	 	&T(a\ast_{1}b+b\ast_{1}a)=a\ast_{1}T(b)+b\ast_{1}T(a)+p(b)Q(a)+p(a)Q(b),\\
	 	&(Q(a)+P(a))\circ_{1}b+(\tau(a)+\omega(a))S(b)=a\ast_{1}S(b)+P(a\circ_{1}b)+q(b)Q(a),\\
	 	&T(Q(a)+P(a))+a_{2}(\tau(a)+\omega(a))=a\ast_{1}a_{2}+k_{2} Q(a)+P(T(a))+a_1p(a),\\
	 	&a_{1}\circ_{1}a+k_{1}S(a)=P(S(a))+a_{1}q(a),\\
	 	&T(a_{1})+k_{1}a_{2}=P(a_{2})+k_{2} a_{1}.\label{74}
	 \end{align}
	 Denote the set of all flag datums of $(A,\ast_{1},\circ_{1})$ by $\mathcal{F}(A)$.
	\end{defi}
\begin{pro}\label{82}
	Let $(A,\ast_{1},\circ_{1})$ be a pre-Poisson algebra and $V$ a vector space of dimension 1 
with a basis $\{x\}$. Then there exists a bijection between the set $\mathcal{A}(A,V)$ of all 
extending structures of $A$ through $V$ and the set of all flag datums
$\mathcal{F}(A)$. The pre-Poisson extending structure 
$\Omega(A,V)=(l_{1},r_{1},l_{2},r_{2},f,\ast_{2},\rho_1,\mu_1,\rho_2,\mu_2,g,\circ_{2})$ corresponding to $(a_{1},k_{1},\tau,\omega,P,Q,a_{2},k_{2},p,q$,
$S,T)\in\mathcal{F}(A)$ is given by
	\begin{align}
		&l_{1}(a)x=\tau(a)x,~r_{1}(a)x=\omega(a)x,\\
		&l_{2}(x)a=P(a),~r_{2}(x)a=Q(a),\\
		&f(x,x)=a_{1},~x\ast_{2}x=k_{1}x,\\
		&\rho_1(a)x=p(a)x,~\mu_1(a)x=q(a)x,\\
		&\rho_2(x)a=S(a),~\mu_2(x)a=T(a),\\
		&g(x,x)=a_{2},~x\circ_{2} x=k_{2} x.
	\end{align}
The corresponding unified product $A\natural V$ is given by
\begin{align}
&(a,x)\ast(b,x):=(a\ast_{1}b+Q(a)+P(b)+a_{1},\tau(a)x+\omega(b)x+k_{1}x),\\
&(a,x)\circ(b,x):=(a\circ_{1}b+T(a)+S(b)+a_{2},p(a)x+q(b)x+k_{2} x),
\end{align}
and the pre-Poisson algebra $A\natural {x}$ is generated by the following relations
\begin{align}
&x\ast x=a_{1}+k_{1}x,~a\ast x=Q(a)+\tau(a)x,~x\ast a=P(a)+\omega(a
)x,\\
&x\circ x=a_{2}+k_{2} x,~a\circ x=T(a)+p(a)x,~x\circ a=S(a)+q(a)x,
\end{align}
for all $a\in A$.
\begin{proof}
	Assume that $\Omega(A,V)=(l_{1},r_{1},l_{2},r_{2},f,\ast_{2},\rho_1,\mu_1,\rho_2,\mu_2,g,\circ_{2})$
is a pre-Poisson extending structure. Thanks to $dim_{k}(V)=1$, we can put
	 \begin{align*}
	 	&x\ast x=a_{1}+k_{1}x,~a\ast x=Q(a)+\tau(a)x,~x\ast a=P(a)+\omega(a
	 	)x,\\
	 	&x\circ x=a_{2}+k_{2} x,~a\circ x=T(a)+p(a)x,~x\circ a=S(a)+q(a)x,
	 \end{align*}
	 where $a_{1},a_{2}\in A$, $k_{1},k_{2}\in k$, $\tau:A\longrightarrow k$, $\omega:A\longrightarrow k$, 
$p:A\longrightarrow k$ and $q:A\longrightarrow k$, $P,Q,S$ and $T:A\longrightarrow A$ are linear maps. 
By Proposition \ref{71}, Proposition \ref{70} and a computation, we obtain that $\Omega(A,V)=(l_{1},r_{1},l_{2},r_{2},f,\ast_{2},\rho_1,\mu_1,\rho_2,\mu_2,g,\circ_{2})$ 
 is a pre-Poisson extending structure if and only if Eqs. \eqref{73}-\eqref{74} hold. The proof is finished.
\end{proof}
\end{pro}
\begin{thm}
	Let $(A,\ast_{1},\circ_{1})$ be a pre-Poisson algebra of codimension 1 in a vector space $E$ and $V$ a 
complement of $A$ in $E$ with basis $\{x\}$. 
Then there is a bijection
	\begin{align}
		Extd(E, A) \cong H^2_{P}(V, A) \cong \mathcal{F}(A)/ \equiv,
	\end{align}
	where $\equiv$ is an equivalent relation on the set $\mathcal{F}(A)$ of all pre-Poisson flag datums of $(A,\ast_{1},\circ_{1})$ defined as follows: $(a_{1},k_{1},\tau,\omega,P,Q,a_{2},k_{2},p,q,S,T) \equiv(a'_{1},k'_{1},\tau',\omega',P',Q',a'_{2},k_{2}',p',q',S',T')$ 
if and only if $\tau=\tau',\omega=\omega',p=p',q=q'$ and there is a pair $(\delta,\epsilon) \in A\times k$ 
satisfying the following conditions,
	\begin{align*}
		&\tau(a)\delta=a\ast_{1}\delta-Q(a)+Q'(a)\epsilon,\\
		&\omega(a)\delta=\delta\ast_{1}a-P(a)+P'(a)\epsilon,\\
		&k_{1}=k'_{1}\epsilon+\tau'(\delta)+\omega'(\delta),\\
		&k_{1}\delta=\delta\ast_{1}\delta+P'(\delta)\epsilon+Q'(\delta)\epsilon+a'_{1}\epsilon^{2}-a_{1},\\
		&p(a)\delta=a\circ_{1}\delta-T(a)+T'(a)\epsilon,\\
		&q(a)\delta=\delta\circ_{1}a-S(a)+S'(a)\epsilon,\\
		&k_{2}=k_{2}'\epsilon+p'(\delta)+q'(\delta),\\
		&k_{2}\delta=\delta\circ_{1}\delta+S'(\delta)\epsilon+T'(\delta)\epsilon+a'_{2}\epsilon^{2}-a_{2},
	\end{align*}
	for all $a\in A$. The bijection between $\mathcal{F}(A)/ \equiv$ and $Extd(E, A)$ is given by
	\begin{align*}
		\overline{(a_{1},k_{1},\tau,\omega,P,Q,a_{2},k_{2},p,q,S,T)} \longmapsto A \natural_{(a_{1},k_{1},\tau,\omega,P,Q,a_{2},k_{2},p,q,S,T)} V,
	\end{align*}  
	where $\overline{(a_{1},k_{1},\tau,\omega,P,Q,a_{2},k_{2},p,q,S,T)}$ is 
the equivalent class of $(a_{1},k_{1},\tau,\omega,P,Q,a_{2},k_{2},p$,
$q,S,T)$ under $\equiv$.
\end{thm}
\begin{proof}
	Let $(a_{1},k_{1},\tau,\omega,P,Q,a_{2},k_{2},p,q,S,T),~(a'_{1},k'_{1},\tau',\omega',P',Q',a'_{2},k_{2}',p',q',S',T')\in\mathcal{F}(A)$ 
and $\Omega(A,V),~\Omega'(A,V)$ be the corresponding pre-Poisson extending structures respectively.
Due to the fact that $dim(V)=1$, we can assume that $x$ is a basis of $V$. Then we can set $\zeta,\eta$ in Lemma \ref{63} as follows
	\begin{align*}
		\zeta(x)=\delta,~\eta(x)=\epsilon x,
	\end{align*}
	where $\delta\in A$ and $\epsilon\in k$. Then the conclusion follows directly from Lemma \ref{63}, Theorem \ref{13} and Proposition \ref{82}.
\end{proof}

\section{Pre-Poisson bialgebras and Yang-Baxter equations in pre-Poisson algebras}

In this section, we investigate a class of special matched pairs of pre-Poisson algebras. 
We introduce the notions of pre-Poisson bialgebras and quadratic pre-Poisson algebras. Furthermore, we establish the equivalence among
 pre-Poisson bialgebras, quadratic pre-Poisson algebras
 and certain matched pairs of pre-Poisson algebras. We study
coboundary pre-Poisson bialgebras and our study leads to the pre-Poisson Yang-Baxter equation (PPYBE) in  pre-Poisson algebras,
which is an analogue of the classical Yang-Baxter equation. A symmetric solution of
PPYBE gives a pre-Poisson bialgebra.

\subsection{Zinbiel bialgebras and pre-Lie bialgebras}
	We begin with recalling the bialgebra theories for Zinbiel algebras and pre-Lie bialgebras \cite{01}.
Since a Zinbiel algebra is a special dendriform algebra, one can turn to \cite{02} for the Zinbiel bialgebra theory.  
 
\begin{defi} 
	\begin{enumerate}
		\item 	A {\bf Zinbiel coalgebra} is a vector space $A$ together
		with a linear map $\Delta: A \longrightarrow A\otimes A$, called the
		coproduct, such that
		\begin{equation}\label{zca eq1.1}(\Delta\otimes I)\Delta+(\tau\otimes I)(\Delta\otimes I)\Delta=(I\otimes \Delta)\Delta.\end{equation}
		\item A {\bf Zinbiel bialgebra} is a triple $(A,\ast,
		\Delta)$, where $(A,\ast)$ is a Zinbiel algebra and $(A,
		\Delta)$ is a Zinbiel coalgebra such that for all $x,y\in A$,
	 the following compatible conditions hold:
		\begin{align}\label{zba eq1.1}\Delta(y\ast z)+\tau\Delta(y\ast z)
			&=(I\otimes L_{\ast}(y))\Delta(z)+\tau(L_{\ast}(y)\otimes
			I)\Delta(z) \\&=(I\otimes R_{\ast}(z))\Delta(y)+
			(L_{\ast}(y)\otimes I)\Delta(z)+ \tau (I\otimes
			L_{\ast}(y))\Delta(z)\nonumber,\end{align}
		\begin{align}\label{zba eq1.2}&\Delta (x\ast y+y \ast x)
			= (L_{\ast}(x)\otimes I)\Delta(y)+(I\otimes
			L_{\star}(y))\Delta(x),\end{align} where
		$x\star y=x\ast y+y\ast x$
		and $\tau(x\otimes y)=y\otimes x$ for all $x,y\in A$.
	\end{enumerate}
\end{defi}
 
 Recall that a {\bf quadratic Zinbiel algebra} $(A,\ast,\omega)$ is a Zinbiel algebra together with a nondegenerate antisymmetric bilinear form $\omega$
such that the following equation holds:
\begin{equation}\label {ef1}\omega( x\ast y,z)=\omega(y,x\star z),
\forall x,y,z\in A.
\end{equation} 

\begin{thm} \label{Zba} Let $(A, \ast_1)$ be a Zinbiel algebra
equipped with a comultiplication $\Delta:A\longrightarrow A\otimes
A$. Suppose that $\Delta^{*}:A^{*}\otimes A^{*}\longrightarrow
A^{*}$ induces a Zinbiel algebra on $A^{*}$. Put
$\ast_2=\Delta^{*}$. Then the following conditions are equivalent:
\begin{enumerate}
	\item There is a quadratic Zinbiel algebra $(A \oplus A^{*},\ast,\omega)$ such that 
	$(A, \ast_1)$ and $(A^{*}, \ast_2)$ are Zinbiel subalgebras, where the bilinear 
	form $\omega$ on $A \oplus A^{*}$ is given by
	\begin{equation*}\omega(x + a, y + b)= \langle x,b\rangle -\langle a,y\rangle,~\forall~x,y\in A,a,b\in A^{*}.
	\end{equation*}
\item $(\mathcal{C}(A),\mathcal{C}(A^{*}),-L_{\ast_1}^{*},-L_{\ast_2}^{*})$ is
a matched pair of commutative associative algebras.
\item $((A, \ast_1),( A^{*},\ast_{2}), -L_{\star_1}^{*}, R_{\ast_1}^{*}, -L_{\star_2}^{*},
R_{\ast_2}^{*}) $ is a matched pair of Zinbiel algebras.
\item $(A, \ast_{1},\Delta)$ is a Zinbiel bialgebra.
\end{enumerate}
\end{thm}

Let $(A,\ast)$ be a Zinbiel algebra. Suppose that
$r=\sum_i a_i \otimes b_i \in A\otimes A$. Set
\begin{equation*}r_{12}\ast
r_{23}:=\sum_{i,j} a_i\otimes  b_i\ast a_j\otimes
b_j,~~r_{13}\star r_{23}:=\sum_{i,j} a_i\otimes a_j\otimes
b_i\star b_j,~~r_{12}\ast r_{13}:=\sum_{i,j}a_i \ast a_j\otimes b_i\otimes b_j
,\end{equation*}
\begin{equation*}r_{13}\ast r_{12}:=\sum_{i,j} a_i\ast a_j\otimes
b_j\otimes b_i,~~r_{13}\ast r_{23}:=\sum_{i,j} a_i\otimes  a_j\otimes b_i\ast
b_j,~~r_{23}\star
r_{12}:=\sum_{i,j} a_j\otimes  a_i\star b_j\otimes
b_i,\end{equation*}
\begin{equation*}r_{23}\ast r_{12}:=\sum_{i,j} a_j\otimes  a_i\ast b_j\otimes
b_i~~,r_{23}\ast r_{13}:=\sum_{i,j} a_j\otimes
a_i\otimes b_i\ast b_j,~~r_{12}\star r_{13}:=\sum_{i,j}a_i \star a_j\otimes b_i\otimes b_j
,\end{equation*}
\begin{equation*}D(r)=r_{23}\ast r_{12}+r_{23}\ast r_{13}-r_{12}\star r_{13},
D_{1}(r)=r_{13}\ast r_{12}+r_{13}\ast r_{23}-r_{23}\star r_{12},\end{equation*}
\begin{equation*}
D_{2}(r)=r_{12}\ast r_{23}+r_{12}\ast r_{13}-r_{13}\star r_{23}.
\end{equation*}
The equation $D(r)=0$ is called the {\bf $D$-equation} in $A$. If $r$ is symmetric,
then the following equations are equivalent:

(i) $D_{1}(r)=0$;

(ii) $D_{2}(r)=0$;

(iii) $D(r)=0$.

\begin{thm}
Let $(A,\ast)$ be a Zinbiel algebra.
Let $r\in A\otimes A$ and $\Delta:A\rightarrow A\otimes A$ be a linear map given by 
\begin{equation}\label{CB1}\Delta(x)=(I\otimes L_{\star}(x)-L_{\ast}(x)\otimes I)r,\;\forall x\in A.\end{equation}
If $r$ is a symmetric solution of the $D$-equation in $(A,\ast)$, then $(A,\ast,\Delta)$ is a Zinbiel bialgebra.
\end{thm}

\begin{defi}
	\begin{enumerate}
		\item A {\bf pre-Lie coalgebra} is a vector space $A$ together
		with a linear map $\delta: A \longrightarrow A\otimes A$ such that
		\begin{align}
			(\delta\otimes I)\delta-(I\otimes\delta)\delta=\tau_{12}((\delta\otimes I)\delta-(I\otimes\delta)\delta).
		\end{align}
	\item A {\bf pre-Lie bialgebra} is a triple $(A,\circ,
	\delta)$, where $(A,\circ)$ is a pre-Lie algebra and $(A,
	\delta)$ is a pre-Lie coalgebra such that for all $x,y\in A$,
	satisfying the following compatible conditions:
	\begin{align}
		\delta(x\circ y-y\circ x)=&(L_\circ(x) \otimes I)\delta(y)+(I\otimes \mathrm{ad}(x))\delta(y)-(L_\circ(y)\otimes I)\delta(x)-(I\otimes \mathrm{ad}(y))\delta(x),\\
		(I-\tau)\delta(x\circ y)=&(I-\tau)((I\otimes R_\circ(y)\delta(x)+(L_\circ(x)\otimes I)\delta(y)+(I\otimes L_\circ(x)\delta(y)),
	\end{align}where
	$\mathrm{ad}(x)=L_{\circ}(x)-R_{\circ}(x)$.
	\end{enumerate}	
\end{defi}
 
 Recall that a {\bf quadratic pre-Lie algebra} $(A,\circ,\omega)$ is a pre-Lie algebra together with a nondegenerate antisymmetric bilinear form $\omega$
 such that the following equation holds:
 \begin{equation}\label {ef2} \omega( x\circ y,z)=-\omega(y,[x,z]),~~\forall~x,y,z\in A.
 \end{equation} 
 
\begin{thm} Let $(A, \circ_1)$ be a pre-Lie algebra
	equipped with a comultiplication $\delta:A\longrightarrow A\otimes
	A$. Suppose that $\delta^{*}:A^{*}\otimes A^{*}\longrightarrow
	A^{*}$ induces a pre-Lie algebra on $A^{*}$. Put
	$\circ_2=\delta^{*}$. Then the following conditions are equivalent:
	\begin{enumerate}
		\item There is a quadratic pre-Lie algebra $(A \oplus A^{*},\ast,\omega)$ such that 
		$(A, \ast_1)$ and $(A^{*}, \ast_2)$ are Zinbiel subalgebras, where the bilinear 
		form $\omega$ on $A \oplus A^{*}$ is given by
		\begin{equation*}\omega(x + a, y + b)= \langle x,b\rangle -\langle a,y\rangle,~\forall~x,y\in A,a,b\in A^{*}.
		\end{equation*}
		\item $(\frak g(A),\frak g(A^{*}), L_{\circ_1}^{*}, L_{\circ_2}^{*})$ is
		a matched pair of Lie algebras.
		\item $((A, \circ_1),( A^{*},\circ_{2}), \mathrm{ad}_{1}^{*}, -R_{\circ_1}^{*}, \mathrm{ad}_{2}^{*}, -R_{\circ_2}^{*})$ is a matched pair of pre-Lie algebras.
		\item $(A, \circ_{1},\delta)$ is a pre-Lie bialgebra.
	\end{enumerate}
\end{thm}

Let $(A,\circ)$ be a pre-Lie algebra. Suppose that
$r=\sum_i a_i \otimes b_i \in A\otimes A$. Denote
\begin{equation*}r_{12}\circ r_{13}=\sum_{i,j}a_i \circ a_j\otimes b_i\otimes b_j
,~~ r_{12}\circ r_{23}=\sum_{i,j} a_i\otimes b_i\circ a_j\otimes
b_j,\end{equation*}
\begin{equation*}[r_{13}, r_{23}]=\sum_{i,j} a_i\otimes  a_j\otimes [b_i,
b_j],~~S(r)=-r_{12}\circ r_{13}+r_{12}\circ r_{23}+[r_{13}, r_{23}].\end{equation*}
The equation
$S(r)=0$ is called the {\bf $S$-equation} in $A$.

\begin{thm}
	Let $(A,\circ)$ be a pre-Lie algebra.
	Let $r\in A\otimes A$ and $\delta:A\rightarrow A\otimes A$ be a linear map given by 
	\begin{equation}\label{CB2}
		\delta(x)=(I\otimes \mathrm{ad}(x)+L_{\circ}(x)\otimes I)r,\;\forall x\in A.\end{equation}
	If $r$ is a symmetric solution of the $S$-equation in $(A,\circ)$, then $(A,\ast,\Delta)$ is a pre-Lie bialgebra.
\end{thm}

\subsection{Quadratic pre-Poisson algebras and pre-Poisson bialgebras}
\begin{defi} Let $(A, \ast,\circ)$ be a pre-Poisson
algebra. If there exists a nondegenerate antisymmetric bilinear form 
$\omega$ on $A$ such that $(A,\ast,\omega)$ is a quadratic Zinbiel algebra and $(A,\circ,\omega)$ is a quadratic pre-Lie algebra, then we say $(A,\ast,\circ,\omega)$ is a {\bf quadratic pre-Poisson algebra}.
\end{defi}

The following result specializes Example 5.23 in \cite{09} from dendriform algebras into Zinbiel algebras.

\begin{pro}\label{pro:4.8} Let $(A, \ast, \circ,\omega)$ be a quadratic pre-Poisson algebra, and $(A, \star, [ \ , \ ])$ be
its sub-adjacent Poisson algebra. Then $\omega$ is both a Connes cocycle on the commutative
associative algebra $(A, \star)$ and a symplectic form on the Lie algebra $(A,  [ \ , \ ])$. 
Conversely,
suppose that $(A, \star, [ \ , \ ])$ is a Poisson algebra, and $\omega$ is
 both a Connes cocycle on the commutative associative algebra $(A, \star)$
  and a symplectic form on the Lie algebra $(A,  [ \ , \ ])$. Then
there exists a compatible quadratic pre-Poisson algebra $(A, \ast, \circ,\omega)$
 given by Eqs.~(\ref{ef1}) and (\ref{ef2}).\end{pro}


\begin{defi} A {\bf pre-Poisson coalgebra} is a triple
$(A,\Delta,\delta)$, where $(A,\Delta)$ is a Zinbiel coalgebra
and $(A,\delta)$ is a pre-Lie coalgebra and $\Delta$ and
$\delta$ are compatible in the following sense:
\begin{equation}\label{Pc1.1}((\delta-\tau\delta)\otimes I)\Delta=(I\otimes \Delta)\delta-(\tau\otimes I)(I\otimes\delta)\Delta,\end{equation}
\begin{equation}\label{Pc1.2}((\Delta+\tau\Delta)\otimes I)\delta=(I\otimes \delta)\Delta+(\tau\otimes I)(I\otimes\delta)\Delta.\end{equation}
\end{defi}

Let $A$ be a vector space with linear maps $\Delta,\delta:A\rightarrow A\otimes A$, and $\ast,\circ:A^{*}\otimes A^{*}\rightarrow A^{*}$ be the linear duals of $\Delta,\delta$ respectively. By a straightforward computation, $(A,\Delta,\delta)$ is a pre-Poisson coalgebra if and only if $(A^{*},\ast,\circ)$ is a pre-Poisson algebra.

\begin{defi} Let $(A,\ast,\circ)$ be a pre-Poisson algebra.
Suppose that there are two comultiplications
$\Delta,\delta:A\longrightarrow A\otimes A$ such that
$(A,\Delta,\delta)$ is a pre-Poisson coalgebra. If in addition, $(A,
\ast,\Delta)$ is a Zinbiel bialgebra, $(A,\circ, \delta)$ is
a pre-Lie bialgebra and $\Delta,\delta$ satisfy the following
compatible conditions:
\begin{align}\label{ppba eq1.4}&
	\delta(x{\star}y)=(L_{\circ}(x)\otimes I)\Delta(y)+(L_{\circ}(y)\otimes I)\Delta(x)-(I\otimes L_{\star}(x))\delta(y)-(I\otimes L_{\star}(y))\delta(x)
	,\end{align}
\begin{align}\label{ppba eq1.1}&
	\Delta([x, y])=(L_{\circ}(x)\otimes I)\Delta(y)-(I\otimes
	L_{\star}(y))\delta(x)+(L_{\ast}(y)\otimes I)\delta(x)+(I\otimes
	\mathrm{ad}(x) )\Delta(y),\end{align}
\begin{align}\label{ppba eq1.3}
\delta(x\ast y)-\tau\delta(x\ast y)=&(I\otimes
L_{\ast}(x))\delta(y)-\tau(L_{\ast}(x)\otimes
I)\delta(y)-(L_{\circ}(x)\otimes I)\Delta(y)\\&-\tau(I\otimes
L_{\circ}(x))\Delta(y)+(I\otimes R_{\ast}(y))\delta(x)
+\tau(I\otimes R_{\circ}(y))\Delta(x)\nonumber ,\end{align}
\begin{align}\label{ppba eq1.2}
	\Delta(x\circ y)+\tau\Delta(x\circ
	y)=&(L_{\circ}(x)\otimes
	I)\Delta(y)+\tau(L_{\circ}(x)\otimes I)\Delta(y)+(I\otimes
	L_{\circ}(x))\Delta(y)\\&+\tau(I\otimes
	L_{\circ}(x))\Delta(y)-(I\otimes
	R_{\ast}(y))\delta(x)-\tau(I\otimes R_{\ast}(y))\delta(x)\nonumber
	,\end{align}
for all $x,y\in A$, then $(A,\ast,\circ, \Delta,\delta)$ is called
a {\bf pre-Poisson bialgebra}.\end{defi}

\begin{pro}\label{pro:411} \cite{04} Let $(P_{1},\star_{1},[ \ , \ ]_1)$ and
$(P_{2},\star_{2},[ \ , \ ]_2)$ be two Poisson
algebras. Suppose that there exist four linear maps
$\mu_{1},\rho_1:P_1\longrightarrow \hbox{End}(P_2)$
and $\mu_{2},\rho_2:P_2\longrightarrow
\hbox{End}(P_1)$.
Define multiplications on $P_{1}\oplus P_{2}$ by
\begin{equation}\label{eq:140}[(x,a),(y,b)]=([x,y]_{1}+\rho_{2}(a)y-\rho_{2}(b)x,[a,b]_{2}
	+\rho_{1}(x)b-\rho_{1}(y)a),\end{equation}
\begin{equation}\label{eq:141}(x,a)\star
	(y,b)=(x\star_{1}y+\mu_{2}(a)y+\mu_{2}(b)x,a\star_{2}b+\mu_{1}(x)b+\mu_{1}(y)a),\;\forall x,y\in P_{1},a,b\in P_{2}.\end{equation}
Then $(P_{1}\oplus P_{2},\star,[\ ,\ ])$ is a Poisson algebra if and only if the following conditions hold:
\begin{enumerate}
	\item $(P_{1},P_{2},\mu_{1},\mu_{2})$
	is a matched pair of commutative associative algebras.
	\item $(P_{1},P_{2},\rho_{1},\rho_{2})$
	is a matched pair of Lie algebras.
	\item $(P_{2},\mu_{1},\rho_{1})$ is a 
	representation of  $(P_1,\star_{1},[ \ , \ ]_1)$.
	\item $(P_{1},\mu_{2},\rho_{2})$ is a 
	representation of  $(P_2,\star_{2},[ \ , \ ]_2)$.
\item The following compatible conditions hold:
\begin{equation}\label{P1}
	\rho_{2}(a)(x\star_{1}y)=(\rho_{2}(a)x)\star_{1}y
	+x\star_1(\rho_{2}(a)y)-\mu_{2}(\rho_{1}(x)a)y
	-\mu_{2}(\rho_{1}(y)a)x, \end{equation}
\begin{equation}\label{P2}[x,\mu_{2}(a)y]_{1}-\rho_{2}(\mu_{1}(y)a)x
	=\mu_{2}(\rho_{1}(x)a)y-(\rho_{2}(a)x)\star_{1}y
	+\mu_{2}(a)([x,y]_{1}),\end{equation}
\begin{equation}\label{P3}\rho_{1}(x)(a\star_{2}b)=(\rho_{1}(x)a)\star_{2}b
	+a\star_2(\rho_{1}(x)b)-\mu_{1}(\rho_{2}(a)x)b
	-\mu_{1}(\rho_{2}(b)x)a,\end{equation}
\begin{equation}\label{P4}[a,\mu_{1}(x)b]_{2}-\rho_{1}(\mu_{2}(b)x)a
	=\mu_{1}(\rho_{2}(a)x)b-(\rho_{1}(x)a)\star_{2}b
	+\mu_{1}(x)([a,b]_{2}),
\end{equation}
for any $x,y\in P_{1}$ and $a,b\in P_{2}$.
\end{enumerate} 
In this case, 
we denote this Poisson algebra simply by $P_{1}\bowtie P_{2}$ and 
$(P_{1},P_{2},\mu_{1},\rho_{1},\mu_{2},\rho_{2})$
satisfying the above conditions is called a {\bf matched pair of
Poisson algebras}.\end{pro}

\begin{lem}\label{lem:pre to poisson}
	Let $(A_{1}, \ast_1,\circ_1)$ and  $(A_{2}, \ast_2,\circ_2)$ be pre-Poisson algebras, whose sub-adjacent Poisson algebras are $(A_{1},\star_{1},[\ ,\ ]_{1})$ and $(A_{2},\star_{2},[\ ,\ ]_{2})$.
	If $(A_1, A_2,
	l_{1},r_{1},\rho_1,\mu_1,l_{2},r_{2},\rho_2,\mu_2)$
	 is a matched pair of
		pre-Poisson algebras, then $(A_1, A_2,
		l_{1}+r_{1},\rho_1-\mu_1,l_{2}+r_{2},\rho_2-\mu_2)$ is a matched pair of Poisson algebras.
\end{lem}
\begin{proof}
	By Theorem \ref{thm:3.8}, there is a pre-Poisson algebra on $A_{1}\oplus A_{2}$ given by \eqref{ppmp eq1.1} and \eqref{ppmp eq1.2}. The sub-adjacent Poisson algebra is given by 
	{\small
	\begin{eqnarray*}
(x,a)\star(y,b)&=&(x,a)\ast (y,b)+ (y,b)\ast (x,a)\\
&=&(x\star_{1}y+	(l_{2}+r_{2})(a)y+(l_{2}+r_{2})(b)x
,a\star_{2}b+	(l_{1}+r_{1})(x)b+(l_{1}+r_{1})(y)a),\\
\ [(x,a),(y,b)]&=&(x,a)\circ (y,b)- (y,b)\circ (x,a)\\
&=&([x,y]_{1}+(\rho_2-\mu_2)(a)y-(\rho_2-\mu_2)(b)x,
[a,b]_{2}+(\rho_1-\mu_1)(x)b-(\rho_1-\mu_1)(y)a).
	\end{eqnarray*}}By Proposition \ref{pro:411}, $(A_1, A_2,
l_{1}+r_{1},\rho_1-\mu_1,l_{2}+r_{2},\rho_2-\mu_2)$ is a matched pair of Poisson algebras.
\end{proof}

\begin{thm} Let $(A, \ast_1,\circ_1)$ be a pre-Poisson algebra
equipped with two comultiplications $\Delta,\delta:A\longrightarrow
A\otimes A$. Suppose that $\Delta^{*},\delta^{*}:A^{*}\otimes
A^{*}\longrightarrow A^{*}$ induce a pre-Poisson algebra structure
on $A^{*}$. Put $\ast_2=\Delta^{*},\circ_2=\delta^{*}$. Then the
following conditions are equivalent:
\begin{enumerate}
	\item\label{equiva1} There is a quadratic pre-Poisson algebra $(A \oplus A^{*},\ast,\circ,\omega )$ such that 
	$(A, \ast_1,\circ_1)$ and $(A^{*}, \ast_2$,
	$\circ_2)$ are pre-Poisson subalgebras, where the bilinear 
	form $\omega$ on $A \oplus A^{*}$ is given by
	\begin{equation}\label{eq:om}
		\omega(x + a, y + b)= \langle x,b\rangle -\langle a,y\rangle,~\forall~x,y\in A,a,b\in A^{*}.
	\end{equation}
\item\label{equiva2}
$(A, A^{*}, -L_{\ast_1}^{*},L_{\circ_1}^{*},-L_{\ast_2}^{*},L_{\circ_2}^{*}) $ is a matched pair of Poisson algebras.
\item\label{equiva3}
$(A, A^{*}, -L_{\ast_1}^{*}-R_{\ast_1}^{*}, R_{\ast_1}^{*}, L_{\circ_1}^{*}-R_{\circ_1}^{*},-R_{\circ_1}^{*},-L_{\ast_2}^{*}-R_{\ast_2}^{*}, R_{\ast_2}^{*},
L_{\circ_2}^{*}-R_{\circ_2}^{*},-R_{\circ_2}^{*})$ is a matched pair of pre-Poisson algebras.
	\item\label{equiva4} $(A, \ast_{1},\circ_{1},\Delta,\delta)$ is a pre-Poisson bialgebra.
\end{enumerate}
\end{thm}

\begin{proof}
By Lemma \ref{lem:pre to poisson}, we have \eqref{equiva3}$\Longrightarrow$\eqref{equiva2}. Hence
	 it is enough to prove that \eqref{equiva1}$\Longleftrightarrow$\eqref{equiva3},
	\eqref{equiva2}$\Longrightarrow$\eqref{equiva1} and
	\eqref{equiva2}$\Longleftrightarrow$\eqref{equiva4}.

\eqref{equiva1}$\Longleftrightarrow$\eqref{equiva3}. Assume that $(A \oplus A^{*},\ast,\circ,\omega)$ is a quadratic pre-Poisson algebra, 
put $x\circ a=\rho_{1}(x)a+\mu_{2}(a)x$ for all $x\in A,a\in A^{*}$.
For all $x,y\in A,a\in A^{*}$,
\begin{equation*}\omega(a,[x,y])=-\langle a,[x,y]_1\rangle=\langle \mathrm{ad}_{1}^{*}(x)a,y\rangle,\end{equation*}
\begin{equation*}\omega(x\circ a,y)=\omega(\rho_{1}(x)a+\mu_{2}(a)x,y)=-\langle \rho_{1}(x)a,y\rangle.\end{equation*}
Hence, $\rho_{1}(x)=\mathrm{ad}_{1}^{*}(x)=L_{\circ_1}^{*}(x)-R_{\circ_1}^{*}(x)$. Analogously,
\begin{equation*}\mu_{1}(x)=-R_{\circ_1}^{*}(x),~\rho_{2}(a)=L_{\circ_2}^{*}(a)-R_{\circ_2}^{*}(a),~\mu_{2}(a)=-R_{\circ_2}^{*}(a),\end{equation*} and
\begin{equation*}l_{1}(x)=-L_{\ast_1}^{*}(x)-R_{\ast_1}^{*}(x),~r_{1}(x)=R_{\ast_1}^{*}(x),
	l_{2}(a)=-L_{\ast_2}^{*}(a)-R_{\ast_2}^{*}(a),~r_{2}(a)=R_{\ast_2}^{*}(a).\end{equation*} Thus, \eqref{equiva3} holds. The converse part can be
checked straightforwardly.
	
\eqref{equiva2}$\Longrightarrow$\eqref{equiva1} By the assumption, there is a Poisson algebra $(A \oplus A^{*},\star,[\ ,\ ])$ given by \eqref{eq:140} and \eqref{eq:141}, where 
	$\mu_{1}=-L_{\ast_1}^{*},\rho_{1}=L_{\circ_1}^{*},\mu_{2}=-L_{\ast_2}^{*},\rho_{2}=L_{\circ_2}^{*}$. 	By a straightforward computation, we know that 
	$\omega$ given by \eqref{eq:om} is a Connes cocycle on $(A_{1}\oplus A_{2},\star)$ and a symplectic form on $(A_{1}\oplus A_{2},[\ ,\ ])$.
	By Proposition \ref{pro:4.8}, there is a compatible quadratic pre-Poisson algebra $(A\oplus A^{*},\ast,\circ,\omega)$  given by 
	\begin{eqnarray*}
	&&\omega((x+a)\ast (y+b),z+c)=\omega(y+b,(x+a)\star (z+c)),\\
	&&\omega((x+a)\circ (y+b),z+c)=-\omega(y+b,[x+a,z+c]),\;\forall x,y,z\in A, a,b,c\in A^{*}.
	\end{eqnarray*}
	Moreover, we have 
	\begin{eqnarray*}
&&\omega(x\ast y,z+c)=\omega(y,x\star(z+c))=\omega(y,x\star_{1}z-L^{*}_{\ast_{1}}(x)c-L^{*}_{\ast_{2}}(c)x)\\
&&=-\langle y,L^{*}_{\ast_{1}}(x)c\rangle=\langle x\ast_{1}y,c\rangle=\omega(x\ast_{1} y,z+c).
	\end{eqnarray*}
	By the nondegeneracy of $\omega$, we have
	$x\ast_{1} y=x\ast y$, and similarly $x\circ_{1} y=x\circ y$, $a\ast_{2} b=a\ast b$, $a\circ_{2} b=a\circ b$. Hence $(A, \ast_1,\circ_1)$ and $(A^{*}, \ast_2$,
	$\circ_2)$ are pre-Poisson subalgebras.
	
\eqref{equiva2}$\Longleftrightarrow$\eqref{equiva4}. 
We only give detailed proof of $(\ref{P1})\Longleftrightarrow (\ref{ppba eq1.4})$ in the case of
\begin{eqnarray*}
	\mu_{1}=-L_{\ast_{1}}^{*},\rho_{1}=L_{\circ_{1}}^{*},
	\mu_{2}=-L_{\ast_{2}}^{*},\rho_{2}=L_{\circ_{2}}^{*}.
\end{eqnarray*}
The other cases can
be checked similarly. In fact, for all $x,y\in A,a,b\in A^{*}$, we have
\begin{eqnarray*}
&&\langle L^{*}_{\circ_{2}}(a)(x\star y),b\rangle=-\langle x\star_{1}y,a\circ_{2}b\rangle=-\langle \delta(x\star_{1}y),a\otimes b\rangle,\\
&&\langle (L^{*}_{\circ_{2}}(a)x)\star_{1}y,b\rangle=-\langle L^{*}_{\circ_{2}}(a)x,L^{*}_{\star_{1}}(y)b\rangle=\langle x,a\circ_{2}L^{*}_{\star_{1}}(y)b\rangle=-\langle (I\otimes L_{\star_{1}}(y))\delta(x),a\otimes b\rangle,\\
&&\langle x\star_{1}(L^{*}_{\circ_{2}}(a)y),b\rangle=-\langle L^{*}_{\circ_{2}}(a)y,L^{*}_{\star}(x)b\rangle=\langle y, a\circ_{2}L^{*}_{\star_{1}}(x)b\rangle=-\langle (I\otimes L_{\star_{1}}(x))\delta(y),a\otimes b\rangle,\\
&&\langle L^{*}_{\ast_{2}}(L^{*}_{\circ_{1}}(x)a)y,b\rangle=-\langle y,L^{*}_{\circ_{1}}(x)a\ast_{2}b\rangle=\langle (L_{\circ_{1}}(x)\otimes I)\Delta(y),a\otimes b\rangle,\\
&&\langle L^{*}_{\ast_{2}}(L^{*}_{\circ_{1}}(y)a)x,b\rangle=-\langle x,L^{*}_{\circ_{1}}(y)a\ast_{2}b\rangle=\langle (L_{\circ_{1}}(y)\otimes I)\Delta(x),a\otimes b\rangle,
\end{eqnarray*}
 which yields that $(\ref{P1})\Longleftrightarrow (\ref{ppba eq1.4})$.
Similarly, we have
$(\ref{P2})\Longleftrightarrow (\ref{ppba eq1.1})$,
$(\ref{P3})\Longleftrightarrow (\ref{ppba eq1.3})$,
$(\ref{P4})\Longleftrightarrow (\ref{ppba eq1.2})$.
\end{proof}

\subsection{Coboundary Pre-Poisson bialgebras.}

In this section, we consider the coboundary pre-Poisson
bialgebras. 

\begin{defi} A pre-Poisson bialgebra $(A,\ast,\circ,\Delta,\delta)$ is called coboundary if there is some
$r\in A\otimes A$ such that \eqref{CB1} and \eqref{CB2} hold.
\end{defi}

\begin{pro} \label{CNPB}
Let $(A,\ast,\circ,\delta,\Delta)$ be a pre-Poisson bialgebra
and  $r=\sum\limits_{i}a_{i}\otimes b_{i}\in A\otimes A$. Define
$\Delta,\delta:A\longrightarrow A\otimes A$ by (\ref{CB1}) and
(\ref{CB2}) respectively. Then
\begin{enumerate}
	\item \label{CNPB1}
	 Eq.~(\ref{Pc1.1}) holds if and only if
	\begin{align}\label{CA1}
		&(I\otimes L_{\ast}(x)\otimes I-I\otimes I\otimes L_{\star}(x))S(r)+(L_{\circ}(x)\otimes I\otimes I)D_{2}(r)\\
		&+\sum_{j} (I\otimes L_{\circ}(a_{j})\otimes L_{\star}(x)+\mathrm{ad}(a_{j})\otimes I\otimes L_{\star}(x))(((r-\tau(r))\otimes b_{j}) )\nonumber\\
		&-(\mathrm{ad}(x\ast a_{j})\otimes I\otimes I+R_{\circ}(a_{j})L_{\star}(x)\otimes I\otimes I)((r-\tau(r))\otimes b_{j})\nonumber\\
		&-(I\otimes L_{\circ}(x\ast a_{j})\otimes I+R_{\circ}(a_{j})\otimes L_{\ast}(x)\otimes I)((r-\tau(r))\otimes b_j)=0.\nonumber
	\end{align}
		\item \label{CNPB2} Eq.~(\ref{Pc1.2}) holds if and only if
	\begin{align}\label{CA2}
		&(I\otimes I\otimes \mathrm{ad}(x))D_{2}(r)+(L_{\ast}(x)\otimes I\otimes I+ I\otimes L_{\ast}(x)\otimes I)S(r)\\
		&+\sum_{j} (I\otimes L_{\ast}(a_j)\otimes \mathrm{ad}(x)-L_{\star}( a_j)\otimes I\otimes \mathrm{ad}(x))((r-\tau(r))\otimes b_j)\nonumber\\
		&+(R_{\circ}(a_j)L_{\star}(x)\otimes I\otimes I- L_{\star}(x\circ a_j)\otimes I \otimes I)((r-\tau(r))\otimes b_j)\nonumber\\
		&+
		(I\otimes L_{\ast}(x\circ a_j)\otimes I-R_{\circ}(a_i)\otimes L_{\ast}(x)\otimes I)((r-\tau(r))\otimes b_j)=0.\nonumber
		\end{align}
	\item \label{CNPB6} Eqs.~(\ref{ppba eq1.4}) and (\ref{ppba eq1.1}) hold naturally.
	
	\item \label{CNPB4} Eq.~(\ref{ppba eq1.3}) holds if and only if
	\begin{align}\label{CA4}
		&(I\otimes L_{\circ}(x\ast y)-I\otimes L_{\ast}(x)L_{\circ}(y)+L_{\circ}(x\ast y)\otimes I-L_{\circ}(x)L_{\ast}(y)\otimes I \\
		&
	+L_{\circ}(x)\otimes L_{\ast}(y)-L_{\circ}(y)\otimes L_{\ast}(x)
		)(r-\tau(r))=0.\nonumber\end{align}
		
			\item \label{CNPB3} Eq.~(\ref{ppba eq1.2}) holds if and only if
		\begin{align}\label{CA3}&
			(L_{\circ}(x)\otimes L_{\ast}(y)-L_{\ast}(y)\otimes L_{\circ}(x))(r-\tau(r))
			\\&+(\tau+I\otimes I)(I\otimes L_{\circ}(x)L_{\ast}(y)-I\otimes L_{\ast}(x\circ y))(r-\tau(r))=0\nonumber.\end{align}
\end{enumerate}
\end{pro}
\begin{proof} 
	For all $x\in A$, we have
	{\small
	\begin{eqnarray*}
&&(\delta\otimes I)\Delta(x)-(\tau\otimes I)(\delta\otimes I)\Delta(x)-(I\otimes \Delta)\delta(x)+(\tau\otimes I)(I\otimes \delta)\Delta(x)\\
&&=\sum_{i,j}a_{i}\circ a_{j}\otimes b_{j}\otimes x\star b_{i}+a_{j}\otimes [a_{i},b_{j}]\otimes x\star b_{i}-(x\ast a_{i})\circ a_{j}\otimes b_{j}\otimes b_{i}-a_{j}\otimes [x\ast a_{i},b_{j}]\otimes b_{i} \\
&&\ \ -b_{j}\otimes a_{i}\circ a_{j}\otimes x\star b_{i}-[a_{i},b_{j}]\otimes a_{j}\otimes x\star b_{i}+b_{j}\otimes (x\ast a_{i})\circ a_{j}\otimes b_{i}
+[x\ast a_{i},b_{j}]\otimes a_{j}\otimes b_{i}\\
&&\ \ -x\circ a_{i}\otimes a_{j}\otimes b_{i}\star b_{j}+x\circ a_{i}\otimes b_{i}\ast a_{j}\otimes b_{j}-a_{i}\otimes a_{j}\otimes [x,b_{i}]\star b_{j}+a_{i}\otimes [x,b_{i}]\ast a_{j}\otimes b_{j}\\
&&\ \ +(x\star b_{i})\circ a_{j}\otimes a_{i}\otimes b_{j}+a_{j}\otimes a_{i}\otimes [x\star b_{i},b_{j}]-b_{i}\circ a_{j}\otimes x\ast a_{i}\otimes b_{j}-a_{j}\otimes x\ast a_{i}\otimes [b_{i},b_{j}]\\
&&=P(1)+P(2)+P(3),
	\end{eqnarray*}}
	where
	{\small
\begin{eqnarray*}
	P(1)&=&\sum_{i,j} -(x\ast a_{i})\circ a_{j}\otimes b_{j}\otimes b_{i}+[x\ast a_{i},b_{j}]\otimes a_{j}\otimes b_{i}-x\circ a_{i}\otimes a_{j}\otimes b_{i}\star b_{j}\\
	&&+x\circ a_{i}\otimes b_{i}\star a_{j}\otimes b_{j}+(x\star b_{i})\circ a_{j}\otimes a_{i}\otimes b_{j}\\
	&=&(L_{\circ}(x)\otimes I\otimes I)D_{2}(r)+\sum_{i,j} -x\circ (a_{i}\ast a_{j})\otimes b_{i}\otimes b_{j}-(x\ast a_{i})\circ a_{j}\otimes b_{j}\otimes b_{i}\\
	&&+[x\ast a_{i},b_{j}]\otimes a_{j}\otimes b_{i}+(x\star b_{i})\circ a_{j}\otimes a_{i}\otimes b_{j}\\
	&=&(L_{\circ}(x)\otimes I\otimes I)D_{2}(r)+\sum_{i,j} -x\circ (a_{i}\ast a_{j})\otimes b_{i}\otimes b_{j}-(x\ast a_{j})\circ a_{i}\otimes b_{i}\otimes b_{j}\\
	&&+[x\ast a_{j},b_{i}]\otimes a_{i}\otimes b_{j}-[x\ast a_{j},a_{i}]\otimes b_{i}\otimes b_{j}+[x\ast a_{j},a_{i}]\otimes b_{i}\otimes b_{j}\\
	&&+(x\star b_{i})\circ a_{j}\otimes a_{i}\otimes b_{j}+(x\star a_{i})\circ a_{j}\otimes b_{i}\otimes b_{j}-(x\star a_{i})\circ a_{j}\otimes b_{i}\otimes b_{j}\\
	&=&(L_{\circ}(x)\otimes I\otimes I)D_{2}(r)-\sum_{j} (\mathrm{ad}(x\ast a_{j})\otimes I\otimes I+R_{\circ}(a_{j})L_{\star}(x)\otimes I\otimes I)((r-\tau(r))\otimes b_{j}),\\
	P(2)&=&\sum_{i,j} -a_{j}\otimes [x\ast a_{i},b_{j}]\otimes b_{i}+b_{j}\otimes (x\ast a_{i})\circ a_{j}\otimes b_{i}+a_{i}\otimes [x,b_{i}]\ast a_{j}\otimes b_{j}\\
	&&-b_{i}\circ a_{j}\otimes x\ast a_{i}\otimes b_{j}-a_{j}\otimes x\ast a_{i}\otimes [b_{i},b_{j}]\\
	&=&\sum_{i,j} -a_{i}\otimes [x\ast a_{j},b_{i}]\otimes b_{j}+b_{i}\otimes (x\ast a_{j})\circ a_{i}\otimes b_{j}-a_{i}\otimes (x\ast a_{j})\circ a_{i}\otimes b_{j}\\
	&&+a_{i}\otimes (x\ast a_{j})\circ a_{i}\otimes b_{j}+a_{i}\otimes [x,b_{i}]\ast a_{j}\otimes b_{j}-b_{i}\circ a_{j}\otimes x\ast a_{i}\otimes b_{j}\\
	&&+a_{i}\circ a_{j}\otimes x\ast b_{i}\otimes b_{j}-a_{i}\circ a_{j}\otimes x\ast b_{i}\otimes b_{j}+a_{i}\otimes x\ast a_{j}\otimes [b_{i},b_j]\\
	&=&\sum_{j} -(I\otimes L_{\circ}(x\ast a_{j})\otimes I+R_{\circ}(a_{j})\otimes L_{\ast}(x)\otimes I)((r-\tau(r))\otimes b_j)\\
	&&+\sum_{i,j}a_{i}\otimes b_{i}\circ (x\ast a_{j})\otimes b_{j} +a_{i}\otimes [x,b_{i}]\ast a_{j}\otimes b_{j}\\
	&&-a_{i}\circ a_{j}\otimes x\ast b_{i}\otimes b_{j}+a_{i}\otimes x\ast a_{j}\otimes [b_{i},b_j]\\
	&=&\sum_{j} -(I\otimes L_{\circ}(x\ast a_{j})\otimes I+R_{\circ}(a_{j})\otimes L_{\ast}(x)\otimes I)((r-\tau(r))\otimes b_j)\\
	&&+\sum_{i,j} a_{i}\otimes x\ast (b_{i}\circ a_{j})\otimes b_{j}-a_{i}\circ a_{j}\otimes x\ast b_{i}\otimes b_{j}+a_{i}\otimes x\ast a_{j}\otimes [b_{i},b_{j}]\\
	&=&(I\otimes L_{\ast}(x)\otimes I)S(r)-\sum_{j} (I\otimes L_{\circ}(x\ast a_{j})\otimes I+R_{\circ}(a_{j})\otimes L_{\ast}(x)\otimes I)((r-\tau(r))\otimes b_j),\\
P(3)&=&\sum_{i,j} a_{i}\circ a_{j}\otimes b_{j}\otimes x\star b_{i}+a_{j}\otimes [a_{i},b_{j}]\otimes x\star b_{i}-b_{j}\otimes a_{i}\circ a_{j}\otimes x\star b_{i}\\
&&-[a_{i},b_{j}]\otimes a_{j}\otimes x\star b_{i}-a_{i}\otimes a_{j}\otimes [x,b_{i}]\star b_{j}+a_{j}\otimes a_{i}\otimes [x\star b_{i},b_{j}]\\
&=&\sum_{i,j} a_{j}\circ a_{i}\otimes b_{i}\otimes x\star b_{j}+a_{i}\otimes [a_{j},b_{i}]\otimes x\star b_{j}-b_{i}\otimes a_{j}\circ a_{i}\otimes x\star b_{j}\\
&&+a_{i}\otimes a_{j}\circ b_{i}\otimes x\star b_{j}-a_{i}\otimes a_{j}\circ b_{i}\otimes x\star b_{j}-[a_{j},b_{i}]\otimes a_{i}\otimes x\star b_{j}\\
&&+[a_{j},a_{i}]\otimes b_{i}\otimes x\star b_{j}-[a_{j},a_{i}]\otimes b_{i}\otimes x\star b_j-a_{i}\otimes a_{j}\otimes x\star [b_{i},b_{j}]\\
&=&\sum_{j} (I\otimes L_{\circ}(a_{j})\otimes L_{\star}(x)+\mathrm{ad}(a_{j})\otimes I\otimes L_{\star}(x))(((r-\tau(r))\otimes b_{j}) )\\
&&+\sum_{i,j} a_{i}\circ a_{j}\otimes b_{i}\otimes x\star b_{j}-a_{i}\otimes b_{i}\circ a_{j}\otimes x\star b_{j}-a_{i}\otimes a_{j}\otimes x\star [b_{i},b_{j}]\\
&=&-(I\otimes I\otimes L_{\star}(x))S(r)+\sum_{j} (I\otimes L_{\circ}(a_{j})\otimes L_{\star}(x)+\mathrm{ad}(a_{j})\otimes I\otimes L_{\star}(x))(((r-\tau(r))\otimes b_{j}) ).
\end{eqnarray*}}Hence Eq.~(\ref{Pc1.1}) holds if and only if Eq.~(\ref{CA1}) holds, which gives rise to \eqref{CNPB1}.
Taking the same procedure, we can prove that \eqref{CNPB2} holds.

Moreover, for all $x,y\in A$,	 we have
\begin{eqnarray*}
	&&\delta(x\ast y)-\tau\delta(x\ast y)-(I\otimes
	L_{\ast}(x))\delta(y)+\tau(L_{\ast}(x)\otimes
	I)\delta(y)\\
	&&+(L_{\circ}(x)\otimes I)\Delta(y)+\tau(I\otimes
	L_{\circ}(x))\Delta(y)-(I\otimes R_{\ast}(y))\delta(x)
	-\tau(I\otimes R_{\circ}(y))\Delta(x)\\
	&&=\sum_{i} (x\ast y)\circ a_{i}\otimes b_{i}+a_{i}\otimes [x\ast y,b_{i}]-b_{i}\otimes (x\ast y)\circ a_{i}-[x\ast y,b_{i}]\otimes a_{i}\\
	&&\ \ -y\circ a_{i}\otimes x\ast b_{i}-a_{i}\otimes x\ast [y,b_{i}]+b_{i}\otimes x\ast (y\circ a_{i})+[y,b_{i}]\otimes x\ast a_{i}\\
	&&\ \ +x\circ a_{i}\otimes y\star b_{i}-x\circ (y\ast a_{i})\otimes b_{i}
	+x\circ (y\star b_{i})\otimes a_{i}-x\circ b_{i}\otimes y\ast a_{i}\\
	&&\ \ -x\circ a_{i}\otimes b_{i}\ast y-a_{i}\otimes [x,b_{i}]\ast y-(x\star b_{i})\circ y\otimes a_{i}+b_{i}\circ y\otimes x\ast a_{i}\\
&&=	\sum_{i} a_{i}\otimes (x\ast y)\circ b_{i}-b_{i}\otimes (x\ast y)\circ a_{i}
+b_{i}\otimes x\ast (y\circ a_{i})-a_{i}\otimes x\ast (y\circ b_{i})\\
&&\ \ -(x\ast y)\circ b_{i}\otimes a_{i}+(x\ast y)\circ a_{i}\otimes b_{i}+x\circ (y\ast b_{i})\otimes a_{i}-x\circ (y\ast a_{i})\otimes b_{i}\\
&&\ \ +y\circ b_{i}\otimes x\ast a_{i}-y\circ a_{i}\otimes x\ast b_{i}-x\circ b_{i}\otimes y\ast a_{i}+x\circ a_{i}\otimes y\ast b_{i}\\
&&=(I\otimes L_{\circ}(x\ast y)-I\otimes L_{\ast}(x)L_{\circ}(y)+L_{\circ}(x\ast y)\otimes I-L_{\circ}(x)L_{\ast}(y)\otimes I\\
&&\ \ +L_{\circ}(x)\otimes L_{\ast}(y)-L_{\circ}(y)\otimes L_{\ast}(x))(r-\tau(r)),
\end{eqnarray*}
which indicates that item \eqref{CNPB4} holds. Taking the same procedure, we can prove that \eqref{CNPB6} and \eqref{CNPB3} hold.
\end{proof}

Let $(A,\ast,\circ)$ be a pre-Poisson algebra and $r\in
A\otimes A$. We say that  $r$ satisfies {\bf the pre-Poisson Yang-Baxter equation} ({\bf PPYBE} in short)
 if $r$ satisfies both the $D$-equation:$$D(r)=r_{23}\ast r_{12}+r_{23}\ast r_{13}-r_{12}\star r_{13}=0,$$
 and the  $S$-equation:
  $$ S(r)=-r_{12}\circ r_{13}+r_{12}\circ r_{23}+[r_{13}, r_{23}]=0.$$
 
 Then by the above results, we reach the following conclusion.
 
\begin{cor}\label{PPYBE} 
	Let $(A,\ast,\circ)$ be a pre-Poisson algebra and $r\in A\otimes A$ be
	a symmetric solution of the  PPYBE in $(A,\ast,\circ)$.
	Then $(A,\ast,\circ)$ is a coboundary pre-Poisson bialgebra, where $\Delta,\delta:A\rightarrow A\otimes A$ are linear maps given by \eqref{CB1} and \eqref{CB2} respectively.
\end{cor}


\begin{center}{\textbf{Acknowledgments}}
\end{center}
This work was supported by the National Natural Science Foundation of
China (12401031), 
the Natural Science
Foundation of Zhejiang Province of China (LY19A010001), the Science
and Technology Planning Project of Zhejiang Province
(2022C01118) and the Postdoctoral Fellowship Program of CPSF (GZC20240755, 2024T005TJ, 2024M761507).

\begin{center} {\textbf{Statements and Declarations}}
\end{center}
 All datasets underlying the conclusions of the paper are available
to readers. No conflict of interest exits in the submission of this
manuscript.


\end {document}